\pdfoutput=1
\documentclass[11pt,a4paper]{article}
\usepackage[T1]{fontenc}
\usepackage{lmodern,microtype}
\usepackage[margin=27mm]{geometry}
\usepackage{amsmath,amssymb,amsthm,mathtools,booktabs,array}
\usepackage[shortlabels]{enumitem}
\usepackage{xcolor,flafter,needspace}
\usepackage{tikz}
\usetikzlibrary{arrows.meta,positioning,fit,backgrounds}
\definecolor{linkblue}{RGB}{30,62,105}
\usepackage[colorlinks=true,allcolors=linkblue]{hyperref}
\usepackage{url}
\newtheorem{theorem}{Theorem}
\newtheorem{lemma}[theorem]{Lemma}
\newtheorem{proposition}[theorem]{Proposition}
\newtheorem{corollary}[theorem]{Corollary}
\theoremstyle{definition}
\newtheorem{question}[theorem]{Question}
\theoremstyle{remark}
\newtheorem{remark}[theorem]{Remark}
\newcommand{\cq}{\chi_q}
\newcommand{\cstar}{\chi_{C^*}}
\newcommand{\thbar}{\bar\vartheta}
\newcommand{\one}{\mathbf{1}}
\newcommand{\R}{\mathbb{R}}
\newcommand{\C}{\mathbb{C}}
\newcommand{\F}{\mathbb{F}}

\newcommand{\cA}{\mathcal{A}}
\newcommand{\cE}{\mathcal{E}}
\newcommand{\cO}{\mathcal{O}}
\newcommand{\cK}{\mathcal{K}}
\DeclareMathOperator{\tr}{tr}
\DeclareMathOperator{\dist}{dist}
\DeclareMathOperator{\su}{sum}
\newcommand{\br}[1]{[#1]}
\setlist{itemsep=2pt,topsep=4pt}
\newcommand{\appendixnumberline}[1]{}
\title{A counterexample to the quantum Hedetniemi conjecture}
\author{Julius A. Zeiss\thanks{\raggedright Institute for Quantum Information, RWTH Aachen University, Aachen, Germany. Email: \href{mailto:jzeiss@physik.rwth-aachen.de}{jzeiss@physik.rwth-aachen.de}.}}
\date{17 September 2026}
\hypersetup{pdftitle={A counterexample to the quantum Hedetniemi conjecture},pdfauthor={Julius A. Zeiss}}
\begin{document}
\maketitle

\begin{abstract}
Godsil, Roberson, \v{S}\'amal and Severini conjectured that the quantum
chromatic number of the categorical product of two graphs equals the minimum
of the quantum chromatic numbers of the factors. We disprove this conjecture:
we construct explicit finite graphs $G,H$ with
\[
 \chi(G\times H)\le1538<1539=\min\{\cq(G),\cq(H)\}.
\]
The graphs are obtained from Zhu's counterexample to Hedetniemi's conjecture
by using a base graph for which the Lov\'asz theta number of the complement,
and not only the fractional chromatic number, is large. The lower bound for
the first factor is the theta bound. For the second factor we adapt Zhu's
argument to projections that do not commute: the step that fixes the colors
of a clique is replaced by identities between operators. Both lower bounds
hold for colorings by projections in an arbitrary nonzero unital
$C^*$-algebra. Hence the conjecture also fails for the
spatial, approximate, commuting-operator and $C^*$-algebraic variants of the
quantum chromatic number. We also give smaller counterexamples certified by
exact integer data. The graph constructions, the certificates and the
counterexample statements in the projective formulation are formalized in
Lean~4.
\end{abstract}

\setcounter{tocdepth}{1}
\tableofcontents

\section{Introduction and main results}\label{sec:intro}

All graphs are finite and simple, with nonempty vertex set. We write
$u\sim v$ for adjacency and $[n]=\{0,\ldots,n-1\}$. The categorical product
$G\times H$ of graphs $G,H$ has vertex set $V(G)\times V(H)$, and
$(x,y)\sim(x',y')$ if and only if $x\sim x'$ and $y\sim y'$.

Hedetniemi's conjecture asserted that
$\chi(G\times H)=\min\{\chi(G),\chi(H)\}$ for all graphs $G,H$. It was
disproved by Shitov~\cite{Shitov}, and Zhu gave much smaller
counterexamples~\cite{Zhu}. Godsil, Roberson, \v{S}\'amal and Severini
studied the same equality for three other graph parameters: the Lov\'asz
theta number of the complement $\thbar$, the vector chromatic number
$\chi_{\mathrm{vec}}$, and the quantum chromatic number $\cq$. They proved it
for $\thbar$ \cite[Theorem~4.6]{GRSS} and conjectured it for the other two
parameters \cite[Section~8]{GRSS}. The equality for $\chi_{\mathrm{vec}}$ was
later proved by Godsil, Roberson, Rooney, \v{S}\'amal and
Varvitsiotis~\cite{GRRSV}. The quantum case,
\begin{equation}\label{eq:conjecture}
 \cq(G\times H)=\min\{\cq(G),\cq(H)\},
\end{equation}
which we call the \emph{quantum Hedetniemi conjecture}, remained open
\cite[Section~2]{ZhuSurvey}. In this paper we show that
\eqref{eq:conjecture} fails. A counterexample to Hedetniemi's conjecture does
not by itself refute \eqref{eq:conjecture}, because the quantum chromatic
numbers of the factors can be smaller than their chromatic numbers.

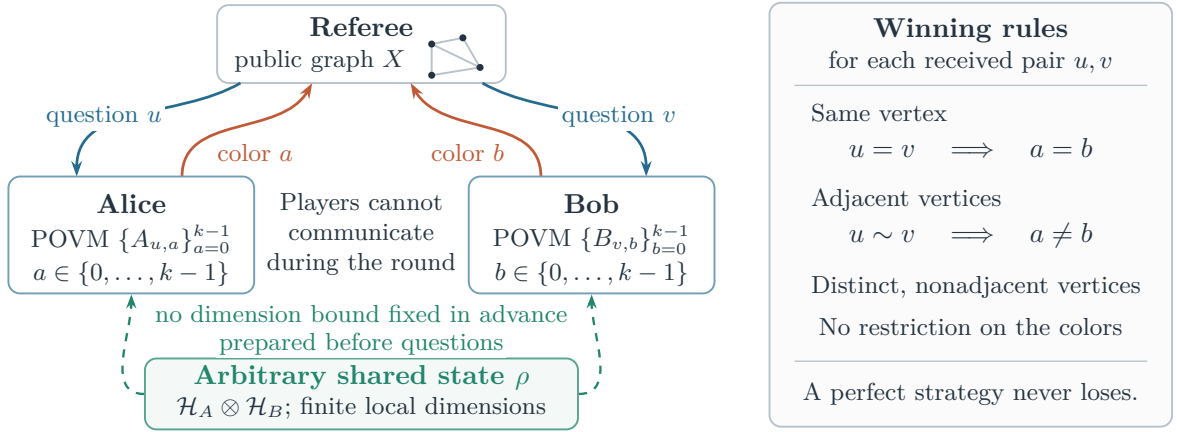
\begin{figure}[tbp]
\centering
\begingroup
\definecolor{question}{HTML}{246B8E}
\definecolor{reply}{HTML}{C05A36}
\definecolor{shared}{HTML}{22856C}
\definecolor{ink}{HTML}{253340}
\definecolor{muted}{HTML}{B8C2CA}
\resizebox{\linewidth}{!}{%
\begin{tikzpicture}[x=1cm,y=1cm,font=\fontsize{10}{12}\selectfont,text=ink,
  box/.style={draw=muted,line width=.7pt,rounded corners=1.5mm,fill=white,align=center},
  questionarrow/.style={-{Stealth[length=2mm,width=1.4mm]},draw=question,line width=1pt},
  replyarrow/.style={-{Stealth[length=2mm,width=1.4mm]},draw=reply,line width=1pt},
  preparearrow/.style={-{Stealth[length=1.8mm,width=1.3mm]},draw=shared,line width=.85pt,dashed},
  small/.style={font=\fontsize{9}{11}\selectfont},
  title/.style={font=\fontsize{10}{12}\selectfont\bfseries}]

\path[use as bounding box] (-.08,-.12) rectangle (15.06,5.78);

\node[box,minimum width=3.7cm,minimum height=1.0cm] (referee) at (4.65,5.05) {};
\node[title] at (4.65,5.27) {Referee};
\node[small] at (4.12,4.84) {public graph $X$};

\begin{scope}[shift={(5.8,4.85)}]
  \draw[muted,line width=.7pt] (-.25,-.2)--(-.25,.2)--(.15,.28)--(.37,-.1)--(-.25,-.2);
  \draw[muted,line width=.7pt] (-.25,.2)--(.37,-.1);
  \foreach \x/\y in {-.25/-.2,-.25/.2,.15/.28,.37/-.1}
    \fill[ink] (\x,\y) circle[radius=.045];
\end{scope}

\draw[questionarrow] (3.1,4.55) to[out=210,in=90] (1.0,3.35);
\node[small,text=question,fill=white,inner sep=1.3pt] at (1.35,4.12) {question $u$};
\draw[replyarrow] (2.35,3.35) to[out=90,in=240] (4.02,4.55);
\node[small,text=reply,fill=white,inner sep=1.3pt] at (3.29,3.66) {color $a$};
\draw[questionarrow] (6.2,4.55) to[out=330,in=90] (8.3,3.35);
\node[small,text=question,fill=white,inner sep=1.3pt] at (7.95,4.12) {question $v$};
\draw[replyarrow] (6.95,3.35) to[out=90,in=300] (5.28,4.55);
\node[small,text=reply,fill=white,inner sep=1.3pt] at (6.01,3.66) {color $b$};

\node[box,draw=question!65,minimum width=3.15cm,minimum height=1.5cm] (alice) at (1.7,2.6) {};
\node[title] at (1.7,3.0) {Alice};
\node[small] at (1.7,2.55) {POVM $\{A_{u,a}\}_{a=0}^{k-1}$};
\node[small] at (1.7,2.10) {$a\in\{0,\ldots,k-1\}$};
\node[box,draw=question!65,minimum width=3.15cm,minimum height=1.5cm] (bob) at (7.6,2.6) {};
\node[title] at (7.6,3.0) {Bob};
\node[small] at (7.6,2.55) {POVM $\{B_{v,b}\}_{b=0}^{k-1}$};
\node[small] at (7.6,2.10) {$b\in\{0,\ldots,k-1\}$};
\node[small,align=center] at (4.65,2.61) {Players cannot\\communicate\\during the round};

\node[box,draw=shared!75,fill=shared!5,minimum width=5.55cm,minimum height=.91cm] (state) at (4.65,.56) {};
\node[title,text=shared] at (4.65,.78) {Arbitrary shared state $\rho$};
\node[small] at (4.65,.38) {$\mathcal H_A\otimes\mathcal H_B$; finite local dimensions};
\draw[preparearrow] (state.west) to[out=180,in=270] (alice.south);
\draw[preparearrow] (state.east) to[out=0,in=270] (bob.south);
\node[small,text=shared,fill=white,inner sep=1pt] at (4.65,1.59) {no dimension bound fixed in advance};
\node[small,text=shared,fill=white,inner sep=1pt] at (4.65,1.20) {prepared before questions};

\draw[box,fill=ink!2] (9.88,.14) rectangle (15.04,5.58);
\node[title] at (12.46,5.23) {Winning rules};
\node[small] at (12.46,4.82) {for each received pair $u,v$};
\draw[muted,line width=.5pt] (10.2,4.52)--(14.72,4.52);
\node[small,anchor=west] at (10.28,4.17) {Same vertex};
\node at (12.46,3.72) {$u=v\quad\Longrightarrow\quad a=b$};
\node[small,anchor=west] at (10.28,3.05) {Adjacent vertices};
\node at (12.46,2.60) {$u\sim v\quad\Longrightarrow\quad a\ne b$};
\node[small,anchor=west] at (10.28,1.93) {Distinct, nonadjacent vertices};
\node[small,align=center] at (12.46,1.42) {No restriction on the colors};
\draw[muted,line width=.5pt] (10.2,.98)--(14.72,.98);
\node[small,align=center] at (12.46,.56) {A perfect strategy never loses.};
\end{tikzpicture}%
}
\endgroup
\caption{Quantum coloring game for $X$ in the finite-dimensional
 tensor-product model. Alice and Bob share an arbitrary density operator
 $\rho$ on $\mathcal H_A\otimes\mathcal H_B$, with both local dimensions
 finite and no dimension bound fixed in advance. On receiving vertices
 $u,v\in V(X)$, they apply local positive operator-valued measurements
 (POVMs) $\{A_{u,a}\}_a$ and $\{B_{v,b}\}_b$, respectively, giving
 $p(a,b\mid u,v)=\operatorname{Tr}[\rho(A_{u,a}\otimes B_{v,b})]$.
 A strategy is perfect if it satisfies the displayed winning rules for
 every input pair. The quantum chromatic number $\chi_q(X)$ is the least
 number $k$ of colors admitting such a strategy.}
\label{fig:quantum-coloring-game}
\end{figure}

\paragraph{Quantum colorings.}
In the coloring game for a graph $X$ with $k$ colors
(Figure~\ref{fig:quantum-coloring-game}), Alice and Bob receive vertices of
$X$ and each answers with a color in $[k]$. They win if they answer equal
vertices with equal colors and adjacent vertices with different colors. They
cannot communicate during the game, but they may share an entangled state. A
strategy is \emph{perfect} if it wins with probability one for every pair of
vertices. The \emph{quantum chromatic number} $\cq(X)$ is the least $k$ for
which a perfect strategy exists in the \emph{finite-dimensional
tensor-product model}, that is, with a state on a tensor product
$\mathcal H_A\otimes\mathcal H_B$ of finite-dimensional Hilbert spaces and
local measurements.

We use the following equivalent formulation, which we call the
\emph{projective formulation}. A \emph{projection} in a unital $C^*$-algebra
$\cA$ with unit $I$ is an element $P$ with $P=P^*=P^2$. Let $k\ge1$ be an
integer. A
\emph{$k$-coloring of $X$ in $\cA$} is a family of projections
$P_{v,a}\in\cA$, $v\in V(X)$, $a\in[k]$, such that
\begin{equation}\label{eq:relations}
 \sum_{a\in[k]}P_{v,a}=I,\qquad
 P_{v,a}P_{v,b}=0\quad(a\ne b),\qquad
 P_{u,a}P_{v,a}=0\quad(u\sim v).
\end{equation}
The second relation follows from the first (fact~\ref{C:pvm} in
Section~\ref{sec:theta}); we list it for convenience. The number $\cq(X)$ is
the least $k$ such that $X$ has a $k$-coloring in the matrix algebra
$M_d(\C)$ for some integer $d\ge1$: a perfect strategy gives such a coloring
by \cite[Proposition~1 and Eq.~(4)]{CMNSW}, and both implications are stated
in \cite[Corollary~2.2]{MR} (for quantum homomorphisms to the complete graph
$K_k$) and in \cite[Theorem~3.2(2)]{HMPS}. The dimension
$d$ is finite, but no upper bound on it is fixed in advance. For every
vertex $v$ the family $(P_{v,a})_{a\in[k]}$ is a projection-valued
measurement in $\cA$; Figure~\ref{fig:projective} summarizes this
formulation.

\begin{figure}[t]
\centering
\begin{tikzpicture}[x=1cm,y=1cm,
  >={Stealth[length=2mm]},
  dot/.style={circle,fill=black,inner sep=1.45pt},
  qarrow/.style={->,line width=.6pt,draw=blue!48!black},
  alink/.style={densely dashed,line width=.55pt,draw=green!40!black},
  measure/.style={draw=blue!48!black,rounded corners=3pt,fill=blue!6,
    align=center,inner sep=5pt,font=\small},
  algebra/.style={draw=green!40!black,rounded corners=3pt,fill=green!7,
    align=center,inner sep=5pt,font=\small},
  panel/.style={draw=black!45,rounded corners=4pt,fill=black!4,
    align=center,inner xsep=8pt,inner ysep=7pt,font=\small}]
  \node[font=\small\bfseries] at (0,5.45) {finite graph $X$};
  \coordinate (x1) at (0,5.00);
  \coordinate (x2) at (-.95,4.30);
  \coordinate (x5) at (.95,4.30);
  \coordinate (u)  at (-.59,3.20);
  \coordinate (v)  at (.59,3.20);
  \draw[line width=.55pt,black!75] (x1)--(x2)--(u) (v)--(x5)--(x1);
  \draw[line width=.9pt] (u)--(v);
  \foreach \p in {x1,x2,x5} \node[dot,fill=black!75] at (\p) {};
  \node[dot] at (u) {};
  \node[dot] at (v) {};
  \node[below right=0pt and 1pt] at (u) {$u$};
  \node[below left=0pt and 1pt] at (v) {$v$};
  \node[font=\footnotesize] at (0,3.44) {$u\sim v$};
  \node[measure,text width=2.9cm] (pu) at (-3.05,1.50)
    {\textbf{measurement at $u$}\\[1pt]
     projections $P_{u,a}$\\[1pt]
     {\footnotesize one per color $a\in\br{k}$}};
  \node[measure,text width=2.9cm] (pv) at (3.05,1.50)
    {\textbf{measurement at $v$}\\[1pt]
     projections $P_{v,a}$\\[1pt]
     {\footnotesize one per color $a\in\br{k}$}};
  \draw[qarrow] (u) -- (pu.north);
  \draw[qarrow] (v) -- (pv.north);
  \node[align=center,font=\footnotesize,text=black!70,text width=2.4cm]
    at (0,1.50)
    {measurements at different vertices need not commute};
  \node[algebra,text width=6.8cm] (alg) at (0,-1.10)
    {\textbf{one nonzero unital $C^*$-algebra $\mathcal A$}\\[1pt]
     {\footnotesize all $P_{v,a}$ lie in $\mathcal A$; $I$ is its unit}\\[2pt]
     {\footnotesize $\cq$:\ $\mathcal A=M_d(\C)$, $d$ finite, no bound
      fixed}\\[1pt]
     {\footnotesize $\cstar$:\ $\mathcal A$ arbitrary}};
  \draw[alink] (pu.south) -- (pu.south |- alg.north);
  \draw[alink] (pv.south) -- (pv.south |- alg.north);
  \node[panel,text width=4.6cm] at (7.55,1.55)
    {\textbf{Coloring relations}\\[1pt]
     {\footnotesize\textcolor{black!60}{for every vertex, edge, and
      color}}\\[3pt]
     \rule{\linewidth}{.4pt}\\[5pt]
     Completeness at each vertex\\[1pt]
     $\sum_{a\in\br{k}}P_{v,a}=I$\\[6pt]
     Same vertex, distinct colors\\[1pt]
     $P_{v,a}P_{v,b}=0\ \ (a\ne b)$\\[6pt]
     Adjacent vertices, same color\\[1pt]
     $P_{u,a}P_{v,a}=0\ \ (u\sim v)$\\[6pt]
     Distinct, nonadjacent\\ vertices\\[1pt]
     {\footnotesize no relation imposed}\\[3pt]
     \rule{\linewidth}{.4pt}\\[5pt]
     The least such $k$ is $\cq(X)$,\\
     resp.\ $\cstar(X)$.};
\end{tikzpicture}
\caption{The projective formulation \eqref{eq:relations} of quantum
coloring, shown at one edge $uv$ of $X$. Every vertex carries a
projection-valued measurement with one outcome per color; all
measurements lie in a single algebra, with no commutativity assumed
between the projections at different vertices. The relations mirror
the rules of the coloring game: agreement on equal vertices, distinct
colors on adjacent vertices, and no constraint on distinct nonadjacent
pairs \cite[Proposition~1]{CMNSW}. For $\cq$ the algebra is a matrix algebra
$M_d(\C)$ with $d$ finite but unrestricted; an arbitrary nonzero unital
$C^*$-algebra defines $\cstar$.}
\label{fig:projective}
\end{figure}

The \emph{$C^*$-chromatic number} $\cstar(X)$ is the least $k$ such that $X$
has a $k$-coloring in some nonzero unital $C^*$-algebra. This parameter was
introduced by Ortiz and Paulsen \cite[Definition~4.9]{OP}, who define it
through projections on a Hilbert space, tacitly assumed to be nonzero, that
satisfy the first and the third relation in \eqref{eq:relations}. It is the parameter of
\cite[Definition~3.12]{HMPS} by \cite[Proposition~3.11]{HMPS}. Our definition
agrees with these, because every nonzero unital $C^*$-algebra has a faithful
unital representation on a nonzero Hilbert space by the Gelfand--Naimark
theorem \cite[Section~3.4]{Murphy}.

An \emph{ordinary $k$-coloring} of $X$ is a map $\varphi\colon V(X)\to[k]$
with $\varphi(u)\ne\varphi(v)$ whenever $u\sim v$; the chromatic number
$\chi(X)$ is the least $k$ for which one exists. An ordinary $k$-coloring
$\varphi$ gives the $k$-coloring $P_{v,a}=\delta_{a,\varphi(v)}$ of $X$ in
$\C$. A $k$-coloring in $\cA$
extends to a $k'$-coloring in $\cA$ for every $k'\ge k$ by setting
$P_{v,a}=0$ for $a\ge k$. Hence $\cstar(X)\le\cq(X)\le\chi(X)$, and $X$ has a
$k$-coloring in some nonzero unital $C^*$-algebra if and only if
$k\ge\cstar(X)$. The inequality $\le$ in \eqref{eq:conjecture} holds because
a $k$-coloring $(P_{x,a})$ of $G$ gives the $k$-coloring
$P_{(x,y),a}=P_{x,a}$ of $G\times H$, and similarly for $H$.

\paragraph{Other models.}
The parameters $\chi_{qs}$, $\chi_{qa}$ and $\chi_{qc}$ are defined like
$\cq$, with perfect strategies in the spatial, approximate and
commuting-operator models; see \cite[Section~2]{KPS} for the four models,
\cite[Definition~2.1(ii)]{PT} for $\chi_{qs}$, and \cite[Section~2]{PSSTW} for
$\chi_{qa}$ and $\chi_{qc}$. The spatial model allows arbitrary
Hilbert spaces $\mathcal H_A,\mathcal H_B$. The approximate model allows
limits of correlations of finite-dimensional strategies. The
commuting-operator model uses a state on one Hilbert space and two families
of measurements such that every operator of Alice commutes with every
operator of Bob. For every graph $X$,
\begin{equation}\label{eq:hierarchy}
 \cstar(X)\le\chi_{qc}(X)\le\chi_{qa}(X)\le\chi_{qs}(X)=\cq(X)\le\chi(X).
\end{equation}
The inequalities between $\chi_{qc}$, $\chi_{qa}$, $\chi_{qs}$, $\cq$ and
$\chi$ are stated in \cite[Section~1]{PSSTW}. The inequality
$\cstar\le\chi_{qc}$ is \cite[Proposition~3.13]{HMPS}: by
\cite[Theorem~3.2(3)]{HMPS}, $\chi_{qc}(X)$ is the least $k$ such that $X$ has a
$k$-coloring in a unital $C^*$-algebra with a faithful tracial state, whereas
no tracial state is required for $\cstar$. The equality $\chi_{qs}=\cq$ holds
because the coloring game is \emph{synchronous}, that is, equal questions
require equal answers: a synchronous game has a perfect strategy in the
spatial model if and only if it has one in the finite-dimensional
tensor-product model \cite[Corollary~3.11]{KPS}.

\paragraph{Main results.}
A real symmetric matrix $M$ indexed by $V(X)$ is \emph{supported on $X$} if
$M_{uv}=0$ whenever $u\ne v$ and $u\not\sim v$. We write
$\su(M)=\sum_{u,v}M_{uv}$. For a graph $F$ and an integer $t\ge1$, the
lexicographic product $F[K_t]$ has vertex set $V(F)\times[t]$, and
$(u,i)\sim(v,j)$ if and only if $u\sim v$, or $u=v$ and $i\ne j$.
Thus $F[K_t]$ is obtained from $F$ by replacing every vertex by a clique
of order $t$.

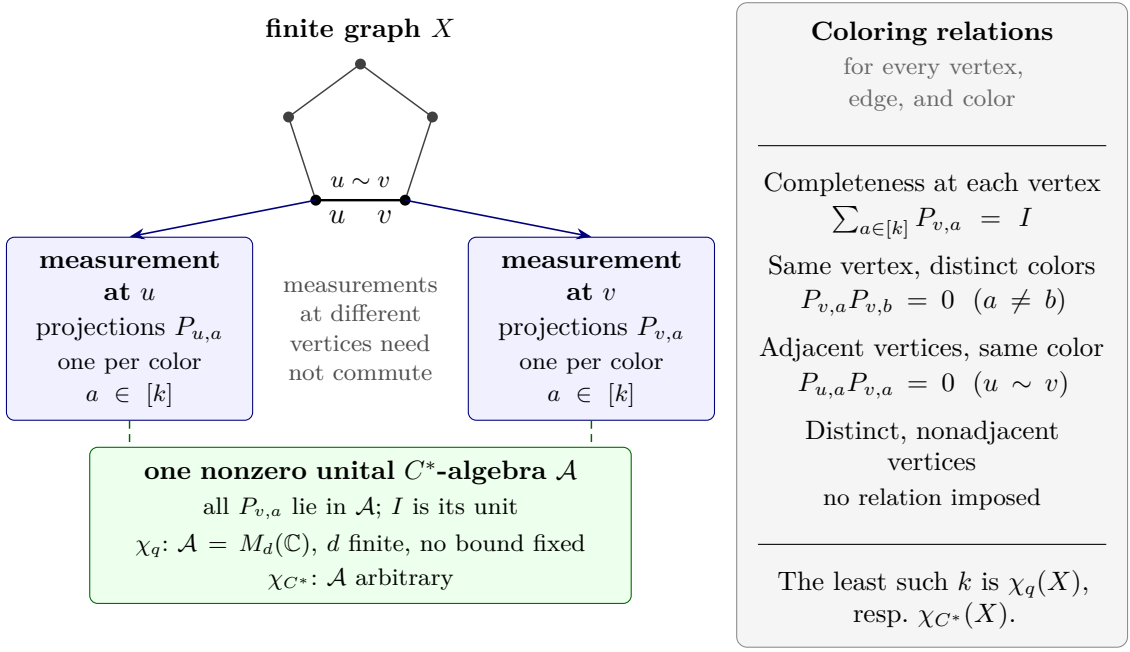
\Needspace{12\baselineskip}
\begin{theorem}\label{thm:main}
Let $F$ be the graph on $1024$ vertices defined in Section~\ref{sec:base},
let $G=F[K_{512}]$, and let $H=H_{512}$ be the graph defined in
Section~\ref{sec:gadget}. Then $|V(G)|=524288$, $|V(H)|=1576451$, and the
following hold.
\begin{enumerate}[(a)]
\item $\chi(G\times H)\le1538$.
\item Neither $G$ nor $H$ has a $1538$-coloring in any nonzero unital
$C^*$-algebra. More precisely, $\cstar(G)\ge1639$ and
$\cstar(H)=\chi(H)=1539$.
\end{enumerate}
Consequently, every graph parameter $\pi$ with
$\cstar(X)\le\pi(X)\le\chi(X)$ for all graphs $X$ satisfies
\begin{equation}\label{eq:main}
 \pi(G\times H)\le1538<1539=\min\{\pi(G),\pi(H)\}.
\end{equation}
By \eqref{eq:hierarchy} this applies to $\cq$, $\chi_{qs}$, $\chi_{qa}$,
$\chi_{qc}$ and $\cstar$. In particular, \eqref{eq:conjecture} fails.
\end{theorem}

Apart from the orders of $G$ and $H$ and the bound $\cstar(G)\ge1639$,
Theorem~\ref{thm:main} follows from the case $t=512$ of the following
statement; its hypotheses are verified for the graph $F$ of
Section~\ref{sec:base} in Lemma~\ref{lem:certificate}.

\begin{theorem}\label{thm:general}
Let $t\ge1$ be an integer, and let $p=2t$ and $c=3t+2$. Let $F$ be a graph
with vertex set $[p]$ that has no closed walk of length three or five. Suppose that there is a real
positive-semidefinite matrix $M$ supported on $F$ with
\begin{equation}\label{eq:theta-hypothesis}
 t\su(M)>c\tr M.
\end{equation}
Let $G=F[K_t]$ and let $H=H_t$ be the graph defined in
Section~\ref{sec:gadget}. Then:
\begin{enumerate}[(a)]
\item $\chi(G\times H)\le c$;
\item $G$ has no $c$-coloring in any nonzero unital $C^*$-algebra;
\item $H$ has no $c$-coloring in any nonzero unital $C^*$-algebra, and
$\chi(H)\le c+1$.
\end{enumerate}
Consequently, every graph parameter $\pi$ with
$\cstar(X)\le\pi(X)\le\chi(X)$ for all graphs $X$ satisfies
$\pi(G\times H)\le c<c+1=\min\{\pi(G),\pi(H)\}$.
\end{theorem}

\begin{remark}\label{rem:hypotheses}
(i) Every closed walk of odd length $\ell$ contains a cycle of odd length at
most $\ell$, and every cycle is a closed walk. Hence $F$ has no closed walk
of length three or five if and only if $F$ has no cycle of length three or
five, that is, $F$ is bipartite or the odd girth of $F$ (the length of a
shortest odd cycle) is at least~$7$. A closed walk $u_0,u_1,u_2,u_0$ of
length three gives the closed walk $u_0,u_1,u_0,u_1,u_2,u_0$ of length five.
Hence it suffices to exclude closed walks of length five; the
formalization~\cite{zeiss-lean} uses the hypothesis in this form.

(ii) The Lov\'asz theta number of the complement of $X$ is
\[
 \thbar(X)=\max\{\su(M):M\succeq0\text{ supported on }X,\ \tr M=1\};
\]
see \cite[Theorem~4]{Lovasz}, applied to the complement of $X$, and
\cite[Section~3]{GRSS}. A matrix $M$ as in \eqref{eq:theta-hypothesis}
satisfies $\tr M>0$, because a positive-semidefinite matrix with trace zero
is the zero matrix; hence it can be scaled to have trace one. Therefore a
matrix $M$ as in
\eqref{eq:theta-hypothesis} exists if and only if $\thbar(F)>c/t=3+2/t$. If
$F$ is bipartite, then no matrix $M$ as in \eqref{eq:theta-hypothesis}
exists: an ordinary $2$-coloring of $F$ gives a $2$-coloring of $F$ in $\C$,
so Lemma~\ref{lem:psd} gives $\su(M)\le2\tr M$ for every real
positive-semidefinite matrix $M$ supported on $F$, and hence
$t\su(M)\le2t\tr M\le c\tr M$, because $\tr M\ge0$ and $2t\le c$. Hence every
graph $F$ as in Theorem~\ref{thm:general} has odd girth at least~$7$.

(iii) In \cite[Section~3]{Zhu}, Zhu constructs a graph $H$ from the order $p$
of a base graph $F$ and an integer $q\ge(p-1)/2$, with $c=3q+2$. The graph
$H_t$ is isomorphic to this graph for $p=2t$ and $q=t$: replace Zhu's index
set $\{1,\ldots,c\}$ by $[c]$ and rename the vertices. Theorem~3 of
\cite{Zhu} assumes that $F$ has odd girth $7$ and that
$\chi_f(F)>3+4/(p-1)$, where $\chi_f$ is the fractional chromatic number. Its
proof uses the second assumption only through the inequality $q\chi_f(F)>c$,
which for $q=t$ reads $\chi_f(F)>c/t$. Theorem~\ref{thm:general} has
\eqref{eq:theta-hypothesis} in place of this inequality, and (b) and (c) in
place of the conclusions $\chi(G)>c$ and $\chi(H)>c$. We have
$\thbar(X)\le\chi_f(X)$ for every graph $X$: apply
\cite[Theorem~10]{Lovasz} to the complement $\overline X$ of $X$; by linear
programming duality, the upper bound $\alpha^*(\overline X)$ given there
equals $\chi_f(X)$. Hence the hypothesis \eqref{eq:theta-hypothesis} implies
$\chi_f(F)>c/t$. The fractional chromatic number cannot be used here, because
it is not a lower bound for $\cq$. Indeed, for $4\mid n$ let $\Omega_n$ be the
graph with vertex set $\{\pm1\}^n$ in which two vectors are adjacent if and
only if they are orthogonal. Then $\cq(\Omega_n)=n$
\cite[Corollary~4.3]{MR}, whereas
$\chi_f(\Omega_n)\ge|V(\Omega_n)|/\alpha(\Omega_n)$, where $\alpha$ denotes
the independence number, grows exponentially in $n$ by the theorem of Frankl
and R\"odl~\cite{FR}.
\end{remark}

\begin{remark}\label{rem:theta-one-factor}
A bound by $\thbar$ cannot give the lower bound for both factors. Indeed,
$\thbar(X)\le\cq(X)$ for every graph $X$ \cite[Corollary~4.1]{MR}, and
$\thbar(G\times H)=\min\{\thbar(G),\thbar(H)\}$ \cite[Theorem~4.6]{GRSS}.
Hence $\chi(G\times H)\le c$ implies
$\min\{\thbar(G),\thbar(H)\}\le\cq(G\times H)\le c$. In
Theorem~\ref{thm:main} we have $\thbar(G)\ge8192/5>1538$ by
\eqref{eq:thetaG}, and therefore $\thbar(H)\le1538<1539=\cq(H)$. For
the same reason \eqref{eq:conjecture} holds whenever $\cq=\thbar$ for both
factors \cite[Theorem~7.3]{GRSS}.
\end{remark}

\paragraph{Overview of the proof.}
A counterexample consists of graphs $G,H$ and an integer $c$ with
$\cq(G\times H)\le c$, $\cq(G)>c$ and $\cq(H)>c$. Zhu's construction
\cite{Zhu} provides such graphs with $\chi$ in place of $\cq$. We use the
following case of it (Remark~\ref{rem:hypotheses}(iii)). Let $t\ge1$ be an
integer and $c=3t+2$. The graphs are $G=F[K_t]$ and $H=H_t$, where $F$ is a
suitable graph with $2t$ vertices and without cycles of length three or
five. We keep this form and choose a different graph $F$. Zhu's upper bound
carries over, because $\cq\le\chi$. His lower bounds do not carry over,
because $\cq$ can be smaller than $\chi$. We prove both lower bounds in a
different way.

\emph{The product.} Let $\psi_y\colon V(G)\to[c]$ be a map for every vertex
$y$ of $H$. The map $(x,y)\mapsto\psi_y(x)$ is an ordinary $c$-coloring of
$G\times H$ if and only if $\psi_y(x)\ne\psi_{y'}(x')$ whenever $x\sim x'$
and $y\sim y'$. Zhu defines such maps explicitly. Their definition uses
distances in $F$, and it works because $F$ has no cycle of length three or
five (Section~\ref{sec:product}).

\emph{The first factor.} Recall that $F[K_t]$ is obtained from $F$ by
replacing every vertex by a clique of order $t$. Zhu uses the bound
$\chi(F[K_t])\ge t\chi_f(F)$, where $\chi_f$ is the fractional chromatic
number. This bound does not help here, because $\chi_f$ is not a lower bound
for $\cq$ (Remark~\ref{rem:hypotheses}(iii)). The theta number $\thbar$ is a
lower bound for $\cq$ and for $\cstar$ (Lemma~\ref{lem:psd}). Moreover,
$\cstar(F[K_t])\ge t\,\thbar(F)$ (Remark~\ref{rem:hoffman}). So we need a
graph $F$ with $2t$ vertices and without cycles of length three or five such
that $\thbar(F)>c/t=3+2/t$. Our graph $F$ is a Cayley graph of the group
$\F_2^{10}$ (Section~\ref{sec:base}). It has $1024$ vertices, so $t=512$ and
$c=1538$. It has no cycle of length three or five
(Lemma~\ref{lem:certificate}(a)). It is $22$-regular, and the least
eigenvalue of its adjacency matrix is $-10$. Every $d$-regular graph $X$
whose adjacency matrix has least eigenvalue $\lambda_{\min}<0$ satisfies
$\thbar(X)\ge1+d/|\lambda_{\min}|$ (Remark~\ref{rem:hoffman}). Hence
$\thbar(F)\ge1+22/10=16/5$, and $\cstar(G)\ge512\cdot16/5>1538$.

\emph{The second factor: Zhu's argument.} The theta number cannot give the
lower bound for $H_t$ (Remark~\ref{rem:theta-one-factor}). We describe $H_t$
and Zhu's argument in words; the details are in Section~\ref{sec:gadget}.
The graph $H_t$ contains a clique of $c$ vertices $e_b$, $b\in[c]$, called
\emph{anchors}. Every other vertex $v$ has a \emph{list} $L(v)\subseteq[c]$,
and $v$ is adjacent to exactly those anchors $e_b$ with $b\notin L(v)$.
Consider an ordinary $c$-coloring of $H_t$. The anchors receive all $c$
colors, so we may assume that $e_b$ has color $b$. Then every other vertex
$v$ has a color in $L(v)$. One vertex $f$, the \emph{center}, has the list
$[2t]$. Let $i\in[2t]$ be its color. The graph $H_t$ contains a clique $C_i$
of $2t+1$ vertices. Every vertex of $C_i$ has the list $\Lambda_i$ or a list
$\Lambda_i\cup\{b\}$ with one further color $b$. Here $\Lambda_i\subseteq[c]$
is a set of $2t$ colors with $i\notin\Lambda_i$, and $b\notin[2t]$. Let $z$
be a vertex of $C_i$ with the list $\Lambda_i\cup\{b\}$. It has a neighbor
$h$ with the list $\{i,b\}$, and $h$ is adjacent to $f$. Since $f$ has color
$i$, the vertex $h$ has color $b$. Therefore $z$ does not have color $b$.
Hence all $2t+1$ vertices of $C_i$ have colors in $\Lambda_i$. This is
impossible, because $C_i$ is a clique and $\Lambda_i$ has only $2t$
elements.

\emph{The second factor: colorings in a $C^*$-algebra.} Now let
$(P_{v,a})$ be a $c$-coloring of $H_t$ in a nonzero unital $C^*$-algebra. The
color of a vertex is no longer a number, and we cannot assume that $e_b$ has
color $b$. Put $Q_{b,a}=P_{e_b,a}$. Every entry of the $c\times c$ matrix
$(Q_{b,a})$ is a projection, and every row and every column sums to $I$
(Lemma~\ref{lem:effects}(a)). Such a matrix is called a magic unitary. For an
ordinary coloring it is a permutation matrix. We replace the statement
``$v$ has the color of the anchor $e_b$'' by the operator
\[
 D_{v,b}=\sum_{a\in[c]}Q_{b,a}P_{v,a}Q_{b,a}.
\]
This operator and $I-D_{v,b}$ are positive elements of the algebra, but
$D_{v,b}$ need not be a projection. In place of the two deductions about $h$
and $z$ in Zhu's argument we use two identities
(Lemma~\ref{lem:effects}(d) and~(e)). If $L(v)=\{b,b'\}$ with $b\ne b'$,
then $D_{v,b}+D_{v,b'}=I$: the vertex $v$ has the color of $e_b$ or of
$e_{b'}$. If $u\sim v$ and $L(u)\cap L(v)=\{b\}$, then $D_{v,b}D_{u,b}=0$:
the vertices $u$ and $v$ do not both have the color of $e_b$. Fix $i\in[2t]$
and put $\Phi=D_{f,i}$. For $h$ and $z$ as above, the identities give
$D_{h,b}\Phi=\Phi$ and then $D_{z,b}\Phi=0$. With this we turn the counting
argument for the clique $C_i$ into a statement about $\Phi$. The $2t+1$
vertices of $C_i$ give the term $(2t+1)\Phi^2$, and the $2t$ elements of
$\Lambda_i$ give the term $2t\Phi^2$. The statement is that
$2t\Phi^2-(2t+1)\Phi^2=-\Phi^2$ is positive. Since $\Phi^2$ is positive as
well, this forces $\Phi=0$. For $b\notin[2t]$ the center $f$ is adjacent to
$e_b$, and $D_{f,b}=0$ follows directly. So the center has the color of no
anchor: $D_{f,b}=0$ for every $b\in[c]$. This gives $P_{f,a}Q_{b,a}=0$ for
all $a$ and $b$. Every column of $(Q_{b,a})$ sums to $I$, so $P_{f,a}=0$ for
every $a$. This contradicts $\sum_aP_{f,a}=I$ (Proposition~\ref{prop:H}).

Both lower bounds use only positivity in a $C^*$-algebra. They need no trace
and no bound on the dimension. Hence they hold for $\cstar$, and therefore
for every parameter in \eqref{eq:hierarchy}.

\paragraph{Outline.}
Section~\ref{sec:theta} proves the theta bound for colorings in a
$C^*$-algebra (Lemma~\ref{lem:psd}). For $\cq$ this bound is
\cite[Corollary~4.1]{MR}, and for $\cstar$ it is
\cite[Proposition~4.10]{OP}, where it is deduced from
a characterization of $\thbar$ by homomorphisms; we give a direct proof.
Section~\ref{sec:base} defines the base graph $F$ of
Theorem~\ref{thm:main} and verifies the hypotheses of
Theorem~\ref{thm:general} for it. Section~\ref{sec:gadget} defines $H_t$ and
proves part~(c) of Theorem~\ref{thm:general}; this is the new argument of the
paper. Section~\ref{sec:product} proves part~(a), following
\cite[Section~3]{Zhu}. Section~\ref{sec:proofs} completes the proofs,
Section~\ref{sec:conclusion} states open questions, and
Appendix~\ref{sec:smaller} gives smaller counterexamples. We make no claim
of minimality.

\section{The theta bound in a \texorpdfstring{$C^*$}{C*}-algebra}\label{sec:theta}

Let $\cA$ be a unital $C^*$-algebra. We write $x\ge0$ if $x$ is positive,
and $x\le y$ if $y-x\ge0$. We use the following standard facts
\cite[Sections~2.1 and~2.2]{Murphy}.
\begin{enumerate}[label=(C\arabic*),ref=(C\arabic*),leftmargin=*]
\item\label{C:cone} $x^*x\ge0$ for every $x\in\cA$. Sums of positive elements
are positive, and $x\ge0$, $-x\ge0$ imply $x=0$.
\item\label{C:conj} If $x\le y$, then $z^*xz\le z^*yz$ for every $z\in\cA$.
\item\label{C:cstar} If $x^*x=0$, then $x=0$, by the $C^*$-identity
$\|x^*x\|=\|x\|^2$.
\item\label{C:scalar} If $\cA\ne0$ and $\lambda\in\R$ satisfies
$\lambda I\ge0$, then $\lambda\ge0$, because the spectrum of $\lambda I$ is
$\{\lambda\}$.
\end{enumerate}
We also use the following consequences.
\begin{enumerate}[label=(C\arabic*),ref=(C\arabic*),leftmargin=*,resume]
\item\label{C:sumzero} If $x_1,\ldots,x_n\ge0$ and $\sum_ix_i=0$, then
$x_i=0$ for every $i$. Indeed, $0\le x_i\le\sum_jx_j=0$ by \ref{C:cone}, so
$x_i\ge0$ and $-x_i\ge0$.
\item\label{C:proj} Every projection $P$ satisfies $0\le P\le I$, because
$P=P^*P$ and $I-P=(I-P)^*(I-P)$. If $P_1,\ldots,P_n$ are projections with
$P_iP_j=0$ for $i\ne j$, then $\sum_iP_i$ is a projection: it is
self-adjoint, and $(\sum_iP_i)^2=\sum_{i,j}P_iP_j=\sum_iP_i$.
\item\label{C:pvm} If $P_1,\ldots,P_n$ are projections with $\sum_iP_i=I$,
then $P_iP_j=0$ for $i\ne j$. Indeed, for fixed $j$,
\[
 \sum_{i\ne j}(P_iP_j)^*(P_iP_j)=\sum_{i\ne j}P_jP_iP_j=P_j(I-P_j)P_j=0,
\]
and \ref{C:sumzero} and \ref{C:cstar} give $P_iP_j=0$.
\end{enumerate}

\begin{lemma}\label{lem:psd}
Let $X$ be a graph, let $k\ge1$ be an integer, and let $M$ be a real
positive-semidefinite matrix supported on $X$. If $X$ has a $k$-coloring in a nonzero unital
$C^*$-algebra, then
\begin{equation}\label{eq:psdbound}
 \su(M)\le k\tr M.
\end{equation}
\end{lemma}
\begin{proof}
Let $(P_{v,a})$ be such a coloring. Choose a real matrix $B$ with
$M=B^{\mathsf T}B$, and set $Y_{v,a}=kP_{v,a}-I$. Each $Y_{v,a}$ is
self-adjoint. For every color $a$,
\[
 \sum_{u,v}M_{uv}Y_{u,a}Y_{v,a}
 =\sum_r\Bigl(\sum_uB_{ru}Y_{u,a}\Bigr)^*
        \Bigl(\sum_vB_{rv}Y_{v,a}\Bigr)\ge0
\]
by \ref{C:cone}. Here the coefficients $B_{ru}$ are real, and the order of
the two operator factors is not changed. We have
$Y_{u,a}Y_{v,a}=k^2P_{u,a}P_{v,a}-kP_{u,a}-kP_{v,a}+I$. Since
$\sum_aP_{u,a}=\sum_aP_{v,a}=I$ and there are $k$ colors,
\[
 \sum_aY_{u,a}Y_{v,a}=k^2\sum_aP_{u,a}P_{v,a}-kI-kI+kI
 =k^2\sum_aP_{u,a}P_{v,a}-kI .
\]
Multiplying by $M_{uv}$ and summing over $u,v$ gives
\begin{align*}
 0&\le\sum_a\sum_{u,v}M_{uv}Y_{u,a}Y_{v,a}
   =k^2\sum_{u,v}M_{uv}\sum_aP_{u,a}P_{v,a}-k\su(M)I\\
  &=k\bigl(k\tr M-\su(M)\bigr)I.
\end{align*}
For the last equality, let $u\ne v$. Then $M_{uv}=0$, or $u\sim v$ and
$P_{u,a}P_{v,a}=0$ for every $a$. The terms with $u=v$ contribute
$\sum_vM_{vv}\sum_aP_{v,a}^2=\tr(M)I$. Since $k>0$, fact~\ref{C:scalar}
gives \eqref{eq:psdbound}.
\end{proof}

The proof uses only the relations \eqref{eq:relations} and positivity. It
uses neither a trace on $\cA$ nor a bound on the dimension.

\begin{corollary}\label{cor:blowup}
Let $F$ be a graph, let $t,k\ge1$ be integers, and let $M$ be a real
positive-semidefinite matrix supported on $F$. If $F[K_t]$ has a
$k$-coloring in a nonzero unital $C^*$-algebra, then
\begin{equation}\label{eq:uniform-bound}
 t\su(M)\le k\tr M.
\end{equation}
\end{corollary}
\begin{proof}
Let $J_t$ be the $t\times t$ all-ones matrix. The matrix $M\otimes J_t$,
with entries $(M\otimes J_t)_{(u,i),(v,j)}=M_{uv}$, is positive semidefinite,
because $M\otimes J_t=W^{\mathsf T}MW$ for the real matrix $W$ with rows
indexed by $V(F)$, columns indexed by $V(F)\times[t]$ and entries
$W_{u,(v,j)}=\delta_{u,v}$. It is supported on $F[K_t]$: if $(u,i)\ne(v,j)$ are not adjacent in
$F[K_t]$, then $u\ne v$ and $u\not\sim v$, so $M_{uv}=0$. Moreover
$\tr(M\otimes J_t)=t\tr M$ and $\su(M\otimes J_t)=t^2\su(M)$. Now apply
Lemma~\ref{lem:psd} to $F[K_t]$ and divide by $t$.
\end{proof}

\begin{remark}\label{rem:hoffman}
By Remark~\ref{rem:hypotheses}(ii), Lemma~\ref{lem:psd} states that
$\thbar(X)\le\cstar(X)$, and Corollary~\ref{cor:blowup} implies
$\cstar(F[K_t])\ge t\,\thbar(F)$. Let $X$ be a $d$-regular graph with
adjacency matrix $A$ and least eigenvalue $\lambda_{\min}<0$. Then
$M=A-\lambda_{\min}I$ is positive semidefinite and supported on $X$, and
$\su(M)/\tr M=1+d/|\lambda_{\min}|$. This number is Hoffman's lower bound
$1-\lambda_{\max}/\lambda_{\min}$ for the chromatic number of
$X$~\cite{Hoffman}, because the largest eigenvalue $\lambda_{\max}$ of $A$
equals $d$; we call it the \emph{Hoffman bound} of $X$. Hence
$\thbar(X)\ge1+d/|\lambda_{\min}|$ by Remark~\ref{rem:hypotheses}(ii), and
$\cstar(X)\ge1+d/|\lambda_{\min}|$ by Lemma~\ref{lem:psd}. This lower bound
is known for $\cq$ and for $\cstar$: combine
$\thbar(X)\ge1+d/|\lambda_{\min}|$ with $\thbar(X)\le\cq(X)$
\cite[Corollary~4.1]{MR} (see \cite{EW}), or with $\thbar(X)\le\cstar(X)$
\cite[Proposition~4.10]{OP}.
\end{remark}

\section{The base graph}\label{sec:base}

Identify $\F_2^{10}$ with the integers $0,\ldots,1023$ through binary
expansion, so that addition in $\F_2^{10}$ is bitwise exclusive-or, written
$\oplus$. Set
\begin{align*}
 S=\{&1,2,4,8,16,31,32,64,128,245,256,371,\\
     &431,457,512,619,697,711,805,861,915,1022\}.
\end{align*}
Let $F$ be the Cayley graph of $\F_2^{10}$ with connection set $S$: its
vertex set is $[1024]$, and $u\sim v$ if and only if $u\oplus v\in S$.
Since $|S|=22$ and $0\notin S$, the graph $F$ is simple and $22$-regular. Let
$A$ be its adjacency matrix, and let $I$ and $J$ be the identity matrix and
the all-ones matrix of order $1024$. Throughout this section
\begin{equation}\label{eq:fibers}
 p=1024,\qquad t=512,\qquad c=3t+2=1538,\qquad G=F[K_t],
\end{equation}
so that $p=2t$ and $|V(G)|=pt=524288$.

For $\ell\ge0$ and $x\in\F_2^{10}$ let $N_\ell(x)$ be the number of
$\ell$-tuples $(s_1,\ldots,s_\ell)\in S^\ell$ with
$s_1\oplus\cdots\oplus s_\ell=x$. Recall that $(A^\ell)_{uv}$ is the
number of walks $(u_0,\ldots,u_\ell)$ of length $\ell$ in $F$ with $u_0=u$
and $u_\ell=v$. For fixed $u$, the map
$(u_0,\ldots,u_\ell)\mapsto(u_0\oplus u_1,\ldots,u_{\ell-1}\oplus u_\ell)$ is
a bijection from the set of walks of length $\ell$ with $u_0=u$ onto
$S^\ell$. Its inverse maps $(s_1,\ldots,s_\ell)$ to the walk with
$u_r=u\oplus s_1\oplus\cdots\oplus s_r$ for $0\le r\le\ell$, which ends at
$u\oplus s_1\oplus\cdots\oplus s_\ell$. Hence
\begin{equation}\label{eq:walks}
 (A^\ell)_{uv}=N_\ell(u\oplus v)\qquad(u,v\in[1024],\ \ell\ge0).
\end{equation}

\begin{lemma}\label{lem:certificate}
The graph $F$ and its adjacency matrix $A$ have the following properties.
\begin{enumerate}[(a)]
\item $F$ has no closed walk of length three or five.
\item $(A+10I)(A+2I)(A-6I)=12J$.
\item $M=A+10I$ is positive semidefinite and supported on $F$, with
$\tr M=10p$ and $\su(M)=32p$.
\end{enumerate}
\end{lemma}
\begin{proof}
The proof uses two finite computations with the set $S$. Each of them
enumerates tuples in $S^\ell$ only for $\ell\le3$, that is, fewer than
$12000$ tuples. All three parts are formalized in Lean~\cite{zeiss-lean}. There
(b) has the equivalent form
$(M+4I)R+240M=12J$, with $M$ and $R$ as in Remark~\ref{rem:sos}, and
$M\succeq0$ is deduced from \eqref{eq:regular-sos}.

\emph{Computation 1.} The sets $S^{(2)}=\{s\oplus s':s,s'\in S\}$ and
$S^{(3)}=\{s\oplus s'\oplus s'':s,s',s''\in S\}$, which have $232$ and $792$
elements, are disjoint.

\emph{Computation 2.} For every $x\in\F_2^{10}$,
\begin{equation}\label{eq:pointwise}
 N_3(x)+6N_2(x)-52N_1(x)-120N_0(x)=12.
\end{equation}
Here $N_0$ is the indicator function of $\{0\}$ and $N_1$ is the indicator
function of $S$.

(a) By \eqref{eq:walks}, $F$ has a closed walk of length $\ell$ if and only
if $N_\ell(0)>0$. If $s_1\oplus s_2\oplus s_3=0$ with $s_i\in S$, then
$0=s_1\oplus s_1\in S^{(2)}$ and $0\in S^{(3)}$. If
$s_1\oplus\cdots\oplus s_5=0$, then
$s_1\oplus s_2=s_3\oplus s_4\oplus s_5\in S^{(2)}\cap S^{(3)}$. Both
contradict Computation~1.

(b) Expanding the product gives $A^3+6A^2-52A-120I$. By \eqref{eq:walks} its
$(u,v)$ entry is the left-hand side of \eqref{eq:pointwise} at
$x=u\oplus v$, which equals $12$.

(c) The matrix $A$ is real symmetric, and $A\one=22\,\one$ for the all-ones
vector $\one$, because $F$ is $22$-regular. Hence the orthogonal complement
$\one^\perp$ is invariant under $A$, and $\R^{1024}$ has an orthogonal basis
consisting of $\one$ and eigenvectors of $A$ in $\one^\perp$. Let
$\xi\in\one^\perp$ be an eigenvector of $A$ with eigenvalue $\lambda$. Then
$J\xi=0$, and (b) gives
$(\lambda+10)(\lambda+2)(\lambda-6)\xi=12J\xi=0$. Hence
$\lambda\in\{-10,-2,6\}$. Thus every eigenvalue of $A$ is at least $-10$, and
$M=A+10I\succeq0$. The matrix $M$ is supported on $F$ because its
off-diagonal entries are those of $A$. Finally $\tr M=10p$ because $A$ has
zero diagonal, and $\su(M)=22p+10p$ because every row of $A$ sums to $22$.
\end{proof}

\begin{corollary}\label{cor:Gbound}
If $G=F[K_{512}]$ has a $k$-coloring in a nonzero unital $C^*$-algebra,
then $k\ge512\cdot32/10=8192/5$. Since $k$ is an integer and $8192/5>1638$,
\begin{equation}\label{eq:Gbound}
 \cstar(G)\ge1639.
\end{equation}
\end{corollary}
\begin{proof}
Apply Corollary~\ref{cor:blowup} with the matrix $M$ of
Lemma~\ref{lem:certificate}(c): $512\cdot32p\le10pk$.
\end{proof}

Only the inequality $\cstar(G)>1538$ is needed below. It holds because
$\su(M)/\tr M=16/5$ exceeds $c/t=3+2/t$. The same matrix gives a lower bound
for $\thbar(G)$: by the proof of Corollary~\ref{cor:blowup}, the matrix
$M\otimes J_{512}$ is positive semidefinite and supported on $G$, and
$\su(M\otimes J_{512})/\tr(M\otimes J_{512})=512\cdot32/10$. Hence, by
Remark~\ref{rem:hypotheses}(ii),
\begin{equation}\label{eq:thetaG}
 \thbar(G)\ge8192/5.
\end{equation}

The least eigenvalue of $A$ is $-10$. Otherwise the proof of
Lemma~\ref{lem:certificate}(c) would show that every eigenvector of $A$ in
$\one^\perp$ has eigenvalue $-2$ or $6$. Then $(A+2I)(A-6I)$ would vanish on
$\one^\perp$ and map $\one$ to $24\cdot16\cdot\one$, so it would equal
$\frac{24\cdot16}{1024}J=\frac38J$, which is impossible for an integer
matrix. Hence $16/5=1+22/10$ is the Hoffman bound of $F$
(Remark~\ref{rem:hoffman}).

\begin{remark}\label{rem:sos}
The identity in Lemma~\ref{lem:certificate}(b) also gives a certificate for
$M\succeq0$ that does not use eigenvectors. In terms of $M=A+10I$ it reads
$M(M-8I)(M-16I)=12J$, or equivalently $(M+4I)R+240M=12J$ with
$R=M^2-28M=A^2-8A-180I$. Since $MJ=32J$, multiplication by $M-32I$ gives
$M(M-8I)(M-16I)(M-32I)=0$. This is the identity
\begin{equation}\label{eq:regular-sos}
 R^2+112M^2=4096M,
\end{equation}
because both $R^2+112M^2-4096M$ and $M(M-8I)(M-16I)(M-32I)$ equal
$M^4-56M^3+896M^2-4096M$. As $R$ and $M$ are real symmetric,
$M=\frac{1}{4096}(R^{\mathsf T}R+112M^{\mathsf T}M)\succeq0$. The
formalization uses \eqref{eq:regular-sos}.
\end{remark}

\begin{remark}\label{rem:golay}
The following identification of $F$ is not used in the proofs. Let
$s_1,\ldots,s_{22}$ be the elements of $S$, let
$\sigma\colon\F_2^{22}\to\F_2^{10}$ be the linear map
$x\mapsto\sum_jx_js_j$, and let $C$ be the kernel of $\sigma$. Thus $C$ is
the binary linear code whose parity-check matrix has the elements of $S$ as
columns, and $\sigma$ is its syndrome map. The set $S$ spans $\F_2^{10}$, so
$\sigma$ is surjective, $\dim C=12$, and $\sigma$ induces a bijection
$\bar\sigma$ from the set of cosets of $C$ onto $\F_2^{10}$. The \emph{coset
graph} of $C$ has the cosets of $C$ as vertices; two cosets are adjacent if
and only if their difference, which is again a coset of $C$, contains a
vector of weight one. This holds if and only if their images under
$\bar\sigma$ differ by an element of $S$. Hence $\bar\sigma$ is an
isomorphism from the coset graph of $C$ onto $F$. A finite computation
shows that $C$ has minimum distance $6$, with exactly $77$ words of weight
$6$, and that there is a linear functional $\eta\colon C\to\F_2$ with
$\eta(w)=1$ for all $77$ words $w$ of weight $6$ such that
$C'=\{(w,\eta(w)):w\in C\}$ is a linear $[23,12,7]$ code. The code $C$ is
obtained from $C'$ by deleting the last coordinate. By the uniqueness of the
binary Golay code~\cite{Pless}, $C'$ is equivalent to the binary Golay code,
so $C$ is equivalent to a punctured binary Golay code, which is called the
truncated binary Golay code in \cite{BCN}. The binary Golay code is cyclic,
so its automorphism group is transitive on the coordinates, and all its
punctured codes are equivalent. Therefore $F$ is isomorphic to the coset
graph of the truncated binary Golay code \cite{Bailey,BCN}. It is
distance-regular with intersection array $\{22,21,20;1,2,6\}$, and the
spectrum of $A$ is $22^{1},6^{330},(-2)^{616},(-10)^{77}$; both statements
can also be verified directly from $S$. Since $F$ is connected and regular
with distinct eigenvalues $22$, $6$, $-2$ and $-10$, the identity in
Lemma~\ref{lem:certificate}(b) is the identity $h(A)=J$ of
Hoffman~\cite{HoffmanPoly} for the Hoffman polynomial
$h(x)=\frac1{12}(x+10)(x+2)(x-6)$ of $F$. The script
\texttt{numerics/check\_base\_graph.py} of the repository~\cite{zeiss-lean}
performs the finite computations of this remark.
\end{remark}

\section{The second factor}\label{sec:gadget}

Let $t\ge1$ be an integer, and let $p=2t$ and $c=3t+2$. Elements of $[c]$ are
called
\emph{labels}. Put
\[
 \Lambda=\{0,\ldots,p\},\qquad\Gamma=\{p+1,\ldots,c-1\},\qquad
 \Lambda_i=\Lambda\setminus\{i\}\quad(i\in[p]).
\]
Thus $|\Lambda_i|=2t$, $|\Gamma|=t+1$, and $[c]$ is the disjoint union of
$\Lambda$ and $\Gamma$. For $i\in[p]$ let $\rho_i\colon[p]\to\Lambda_i$ be
the bijection with $\rho_i(j)=j$ for $j\ne i$ and $\rho_i(i)=p$, and put
\[
 \cE_i=\{\rho_i(j):0\le j<t\},\qquad
 \cO_i=\{\rho_i(t+j):0\le j<t\}.
\]
The sets $\cE_i$ and $\cO_i$ have $t$ elements each and partition
$\Lambda_i$. In the proof that $H_t$ has no $c$-coloring in a nonzero unital
$C^*$-algebra (Proposition~\ref{prop:H}) the sets $\cO_i$ serve only as
index sets of size $t$. Their elements are used as colors in
the ordinary $(c+1)$-coloring of Proposition~\ref{prop:H}, and the maps
$\rho_i$ and the sets $\cE_i$, $\cO_i$ are used in
Section~\ref{sec:product}.

The graph $H_t$ has the following vertices, each with a \emph{list}
$L(v)\subseteq[c]$:
\begin{center}
\begin{tabular}{llll}
\toprule
Vertices & Name & Indices & List $L$\\
\midrule
$e_b$ & anchors & $b\in[c]$ & $\{b\}$\\
$f$ & center & & $[p]$\\
$\ell_{i,o}$ & & $i\in[p]$, $o\in\cO_i$ & $\Lambda_i$\\
$z_{i,b}$ & & $i\in[p]$, $b\in\Gamma$ & $\Lambda_i\cup\{b\}$\\
$h_{i,b}$ & & $i\in[p]$, $b\in\Gamma$ & $\{i,b\}$\\
\bottomrule
\end{tabular}
\end{center}
For $i\in[p]$ let $C_i$ be the set of the $2t+1$ vertices $\ell_{i,o}$
($o\in\cO_i$) and $z_{i,b}$ ($b\in\Gamma$). The edges of $H_t$ are:
\begin{enumerate}[label=(E\arabic*),ref=(E\arabic*),leftmargin=*]
\item\label{E:anchors} $e_be_{b'}$ for all $b\ne b'$ in $[c]$;
\item\label{E:lists} $ve_b$ for every vertex $v$ that is not an anchor and
every $b\in[c]\setminus L(v)$;
\item\label{E:clique} $yy'$ for all $y\ne y'$ in $C_i$, for every $i\in[p]$;
\item\label{E:fh} $fh_{i,b}$ for all $i\in[p]$ and $b\in\Gamma$;
\item\label{E:hz} $h_{i,b}z_{i,b}$ for all $i\in[p]$ and $b\in\Gamma$.
\end{enumerate}
\begin{figure}[htbp]
\centering
\begin{minipage}[t]{.43\textwidth}
\centering
\vspace{0pt}
\textbf{(a) \(F[K_3]\) at an edge \(uv\) of \(F\)}\par\smallskip
\begin{tikzpicture}[x=1.08cm,y=1.08cm,
  dot/.style={circle,fill=black,inner sep=1.45pt},
  every node/.style={font=\small},
  crossedge/.style={draw=blue!48!black,opacity=.50,line width=.45pt}]
  \path[use as bounding box] (-2.65,-1.58) rectangle (2.65,2.76);
  \coordinate (u) at (-1.60,2.25);
  \coordinate (v) at (1.60,2.25);
  \draw[line width=.65pt] (u)--(v);
  \node[dot,label=above:\(u\)] at (u) {};
  \node[dot,label=above:\(v\)] at (v) {};
  \draw[->,line width=.55pt] (-1.60,1.95)--(-1.60,1.16);
  \draw[->,line width=.55pt] (1.60,1.95)--(1.60,1.20);
  \coordinate (u1) at (-2.25,.03);
  \coordinate (u2) at (-1.55,.87);
  \coordinate (u3) at (-1.20,-.56);
  \coordinate (v1) at (1.12,.06);
  \coordinate (v2) at (1.80,.91);
  \coordinate (v3) at (2.38,-.43);
  \foreach \i in {1,2,3}
    \foreach \j in {1,2,3}
      \draw[crossedge] (u\i)--(v\j);
  \draw[line width=.8pt] (u1)--(u2)--(u3)--cycle;
  \draw[line width=.8pt] (v1)--(v2)--(v3)--cycle;
  \foreach \i in {1,2,3}{
    \node[dot] at (u\i) {};
    \node[dot] at (v\i) {};
  }
  \node at (-1.82,-.91) {\(K_3\)};
  \node at (1.80,-.91) {\(K_3\)};
  \node[font=\footnotesize,text=blue!48!black] at (0,-1.40)
    {complete join: all nine cross-edges};
\end{tikzpicture}
\end{minipage}\hfill
\begin{minipage}[t]{.55\textwidth}
\centering
\vspace{0pt}
\textbf{(b) Part of \(H_1\) for \(i=0\)}\par\medskip
\begin{tikzpicture}[x=.98cm,y=.90cm,
  dot/.style={circle,fill=black,inner sep=1.5pt},
  every node/.style={font=\small},
  listlabel/.style={font=\footnotesize,text=black!68}]
  \path[use as bounding box] (-3.20,-1.15) rectangle (3.20,4.25);
  \coordinate (f) at (0,3.60);
  \coordinate (h3) at (-1.35,2.34);
  \coordinate (h4) at (1.35,2.34);
  \coordinate (z3) at (-1.35,.56);
  \coordinate (z4) at (1.35,.56);
  \coordinate (ell) at (0,-.56);
  \fill[blue!6] (z3)--(z4)--(ell)--cycle;
  \draw[line width=.7pt] (f)--(h3) (f)--(h4);
  \draw[line width=.7pt] (h3)--(z3) (h4)--(z4);
  \draw[line width=.8pt] (z3)--(z4)--(ell)--cycle;
  \foreach \v in {f,h3,h4,z3,z4,ell}
    \node[dot] at (\v) {};
  \node[above=3pt] at (f) {\(f:\{0,1\}\)};
  \node[above left=2pt] at (h3) {\(h_{0,3}\)};
  \node[below left=2pt,listlabel] at (h3) {\(\{0,3\}\)};
  \node[above right=2pt] at (h4) {\(h_{0,4}\)};
  \node[below right=2pt,listlabel] at (h4) {\(\{0,4\}\)};
  \node[above left=2pt] at (z3) {\(z_{0,3}\)};
  \node[below left=2pt,listlabel] at (z3) {\(\{1,2,3\}\)};
  \node[above right=2pt] at (z4) {\(z_{0,4}\)};
  \node[below right=2pt,listlabel] at (z4) {\(\{1,2,4\}\)};
  \node[below=3pt] at (ell) {\(\ell_{0,1}:\{1,2\}\)};
  \node[font=\footnotesize] at (0,.19) {clique \(C_0\)};
  \node[draw=black!45,dashed,rounded corners=2pt,
    fill=white,align=center,inner sep=4pt,font=\footnotesize]
    at (0,1.51) {anchor clique\\\(K_5\)};
\end{tikzpicture}
\end{minipage}
\caption{The graphs \(F[K_t]\) and \(H_t\) for small parameters.
(a) In \(F[K_3]\), the two cliques that replace the ends of an edge \(uv\) of
\(F\) are completely joined.
(b) The vertices \(f\), \(h_{0,b}\), \(z_{0,b}\) and \(\ell_{0,1}\) of \(H_1\);
braces give their lists. The three vertices of \(C_0\) form a clique, and the
edges \(h_{0,b}z_{0,b}\) form a matching. Every displayed vertex is adjacent
to every anchor \(e_r\) with \(r\) not in its list; these edges and the
vertices with \(i=1\) are omitted.}
\label{fig:construction}
\end{figure}

Thus the anchors form a clique of order $c$, each $C_i$ is a clique of order
$2t+1$, and by \ref{E:anchors} and \ref{E:lists} every vertex $v$ of $H_t$
satisfies
\begin{equation}\label{eq:list-rule}
 L(v)=\{b\in[c]:v\text{ is not adjacent to }e_b\}.
\end{equation}
The number of vertices is
\begin{equation}\label{eq:Horder}
 |V(H_t)|=c+1+p\bigl(t+2(t+1)\bigr)=(p+1)c+1,
\end{equation}
which equals $1576451$ for $t=512$.
Figure~\ref{fig:construction} illustrates the graphs $F[K_t]$ and $H_t$.

An ordinary $c$-coloring of $H_t$ does not exist, by the following argument
of Zhu~\cite{Zhu}. Suppose that $\varphi$ is an ordinary $c$-coloring of
$H_t$. The $c$ anchors are pairwise adjacent, so they receive all $c$ colors,
and after a permutation of the colors we may assume that every anchor $e_b$
has color $b$. Then $\varphi(v)\in L(v)$ for every vertex $v$ by
\eqref{eq:list-rule}. The center has a color $i\in[p]$. For every
$b\in\Gamma$, the vertex $h_{i,b}$ is adjacent to $f$ and therefore has color
$b$, so $z_{i,b}$ does not have color $b$. Hence all $2t+1$ vertices of the
clique $C_i$ have colors in the set $\Lambda_i$ of size $2t$, a
contradiction.

In a $c$-coloring in a $C^*$-algebra, the anchors do not have fixed colors,
and the first step of this argument is not available. The next lemma
replaces it. The operators $D_{v,b}$ defined there play the role of the event
that $v$ receives the color of the anchor $e_b$. They are positive
contractions (Lemma~\ref{lem:effects}(c)). In general they are not
projections, and $\sum_bD_{v,b}$ need not equal $I$; see the example after
the proof of the lemma.

\begin{lemma}\label{lem:effects}
Let $X$ be a graph and let $c\ge1$ be an integer. Let $e_0,\ldots,e_{c-1}$ be
distinct, pairwise adjacent vertices of $X$. As for $H_t$, we call these
vertices \emph{anchors} and the elements of $[c]$ \emph{labels}, and we
define $L(v)$ for every vertex $v$ of $X$ by \eqref{eq:list-rule}. Let
$(P_{v,a})$ be a $c$-coloring of $X$ in a unital $C^*$-algebra $\cA$; thus
the number of colors equals the number of anchors. Here and in the proof of
Proposition~\ref{prop:H}, the letters $a,a'$ denote colors of the
$c$-coloring $(P_{v,a})$, that is, second indices of the projections
$P_{v,a}$; all other indices in $[c]$ are labels.
For $a,b\in[c]$, $\Theta\subseteq[c]$ and $v\in V(X)$ put
\[
 Q_{b,a}=P_{e_b,a},\qquad
 Q_{\Theta,a}=\sum_{b\in\Theta}Q_{b,a},\qquad
 D_{v,b}=\sum_{a\in[c]}Q_{b,a}P_{v,a}Q_{b,a}.
\]
\begin{enumerate}[(a)]
\item $\sum_aQ_{b,a}=I$ for every $b$, and $\sum_bQ_{b,a}=I$ for every $a$.
Moreover $Q_{b,a}Q_{b,a'}=0$ for $a\ne a'$, and $Q_{b,a}Q_{b',a}=0$ for
$b\ne b'$. Consequently every $Q_{\Theta,a}$ is a projection, and
$Q_{\Theta,a}Q_{\Theta',a}=Q_{\Theta\cap\Theta',a}$ for all
$\Theta,\Theta'\subseteq[c]$.
\item $P_{v,a}Q_{L(v),a}=Q_{L(v),a}P_{v,a}=P_{v,a}$ for every vertex $v$ and
every color $a$.
\item $D_{v,b}=\sum_{a\in[c]}(P_{v,a}Q_{b,a})^*(P_{v,a}Q_{b,a})$ and
$0\le D_{v,b}\le I$ for every vertex $v$ and every label $b$.
\item If $L(v)=\{b,b'\}$ with $b\ne b'$, then $D_{v,b}+D_{v,b'}=I$.
\item If $u\sim v$ and $L(u)\cap L(v)=\{b\}$, then $D_{v,b}D_{u,b}=0$.
\end{enumerate}
\end{lemma}

Part (a) states that $(Q_{b,a})_{b,a\in[c]}$ is a magic unitary, also called
a quantum permutation matrix.

\begin{proof}
(a) The relations for fixed $b$ are part of \eqref{eq:relations} for the
vertex $e_b$. Let $a\in[c]$. For $b\ne b'$ the anchors $e_b,e_{b'}$ are
adjacent, so $Q_{b,a}Q_{b',a}=0$. By \ref{C:proj}, $U_a=\sum_bQ_{b,a}$ is a
projection and $I-U_a\ge0$. Since the number of anchors equals the number of
colors,
\[
 \sum_{a\in[c]}(I-U_a)=cI-\sum_{b\in[c]}\sum_{a\in[c]}Q_{b,a}=cI-cI=0.
\]
By \ref{C:sumzero}, $U_a=I$ for every $a$. For fixed $a$ the projections
$Q_{b,a}$, $b\in[c]$, are pairwise orthogonal. Hence $Q_{\Theta,a}$ is a
projection by \ref{C:proj}, and
$Q_{\Theta,a}Q_{\Theta',a}
=\sum_{b\in\Theta}\sum_{b'\in\Theta'}Q_{b,a}Q_{b',a}
=Q_{\Theta\cap\Theta',a}$.

(b) If $b\notin L(v)$, then $v\sim e_b$ by \eqref{eq:list-rule}, so
$P_{v,a}Q_{b,a}=0$ and $Q_{b,a}P_{v,a}=0$. With (a) we get
$P_{v,a}=P_{v,a}\sum_bQ_{b,a}=P_{v,a}Q_{L(v),a}$, and in the same way
$P_{v,a}=Q_{L(v),a}P_{v,a}$.

(c) Since $P_{v,a}=P_{v,a}^*P_{v,a}$ and $Q_{b,a}$ is self-adjoint,
\[
 Q_{b,a}P_{v,a}Q_{b,a}=(P_{v,a}Q_{b,a})^*(P_{v,a}Q_{b,a}),
\]
and summing over $a$ gives the first assertion. By \ref{C:proj} and
\ref{C:conj},
$0\le Q_{b,a}P_{v,a}Q_{b,a}\le Q_{b,a}$. Summing over $a$ and using (a) gives
$0\le D_{v,b}\le I$.

(d) Abbreviate $P_a=P_{v,a}$, $Q_a=Q_{b,a}$, $Q'_a=Q_{b',a}$ and
$\bar Q_a=Q_a+Q'_a=Q_{L(v),a}$. By (a), $\bar Q_a$ is a projection,
$Q_a\bar Q_a=Q_a$, and $Q_aQ'_a=0$. By (b),
$\bar Q_aP_a=P_a\bar Q_a=P_a$. Hence $P'_a=\bar Q_a-P_a$ is self-adjoint
with
\[
 (P'_a)^2=\bar Q_a^2-\bar Q_aP_a-P_a\bar Q_a+P_a^2=\bar Q_a-P_a=P'_a ,
\]
so $P'_a$ is a projection. Moreover
$\sum_aP'_a=\sum_aQ_a+\sum_aQ'_a-\sum_aP_a=I$ by (a), and therefore
$P'_aP'_{a'}=0$ for $a\ne a'$ by \ref{C:pvm}. Put $Z=\sum_aQ'_aP_a$. Since
$\sum_a\bar Q_aP_a=\sum_aP_a=I$ and $Q_aP'_a=Q_a-Q_aP_a$, we have
\[
 Z=I-\sum_aQ_aP_a=\sum_aQ_aP'_a .
\]
The first expression for $Z$ and $P_aP_{a'}=\delta_{a,a'}P_a$ give
\[
 ZZ^*=\sum_{a,a'}Q'_aP_aP_{a'}Q'_{a'}=\sum_aQ'_aP_aQ'_a=D_{v,b'}.
\]
The last expression for $Z$ and $P'_aP'_{a'}=\delta_{a,a'}P'_a$ give
\[
 ZZ^*=\sum_{a,a'}Q_aP'_aP'_{a'}Q_{a'}=\sum_aQ_aP'_aQ_a
     =\sum_a(Q_a-Q_aP_aQ_a)=I-D_{v,b}.
\]
Hence $D_{v,b}+D_{v,b'}=I$.

(e) By (a), $Q_{L(v),a}Q_{L(u),a}=Q_{\{b\},a}=Q_{b,a}$. With (b) and
$u\sim v$,
\[
 P_{v,a}Q_{b,a}P_{u,a}
 =P_{v,a}Q_{L(v),a}Q_{L(u),a}P_{u,a}=P_{v,a}P_{u,a}=0 .
\]
Since $Q_{b,a}Q_{b,a'}=\delta_{a,a'}Q_{b,a}$,
\[
 D_{v,b}D_{u,b}
 =\sum_{a,a'}Q_{b,a}P_{v,a}Q_{b,a}Q_{b,a'}P_{u,a'}Q_{b,a'}
 =\sum_aQ_{b,a}\bigl(P_{v,a}Q_{b,a}P_{u,a}\bigr)Q_{b,a}=0.\qedhere
\]
\end{proof}

The operators $D_{v,b}$ need not be projections, and $\sum_bD_{v,b}$ need
not equal $I$. For example, let $c=3$, and let $X$ consist of three pairwise
adjacent vertices $e_0,e_1,e_2$ and a vertex $v$ without neighbors, so that
$L(v)=[3]$. Let $\cA=M_2(\C)$, and let $\Pi$ and $\Pi'$ be the orthogonal
projections onto the lines spanned by $(1,0)$ and $(1,1)$. The families
\begin{align*}
 (P_{e_0,a})_{a\in[3]}&=(\Pi,I-\Pi,0),&
 (P_{e_1,a})_{a\in[3]}&=(I-\Pi,\Pi,0),\\
 (P_{e_2,a})_{a\in[3]}&=(0,0,I),&
 (P_{v,a})_{a\in[3]}&=(I-\Pi',0,\Pi')
\end{align*}
form a $3$-coloring of $X$ in $\cA$. Since $\Pi(I-\Pi')\Pi=\frac12\Pi$ and
$(I-\Pi)(I-\Pi')(I-\Pi)=\frac12(I-\Pi)$, we get $D_{v,0}=\frac12\Pi$,
$D_{v,1}=\frac12(I-\Pi)$ and $D_{v,2}=\Pi'$. Thus $D_{v,0}$ is not a
projection, and $\sum_bD_{v,b}=\frac12I+\Pi'\ne I$.

\begin{proposition}\label{prop:H}
Let $t\ge1$ be an integer and $c=3t+2$. The graph $H_t$ has no $c$-coloring in
any nonzero
unital $C^*$-algebra, and $\chi(H_t)\le c+1$. Consequently
$\cstar(H_t)=\cq(H_t)=\chi(H_t)=c+1$.
\end{proposition}
\begin{proof}
Suppose that $(P_{v,a})$ is a $c$-coloring of $H_t$ in a nonzero unital
$C^*$-algebra. We use Lemma~\ref{lem:effects} with the anchors
$e_0,\ldots,e_{c-1}$; by \eqref{eq:list-rule} the sets $L(v)$ of the lemma
are the lists of $H_t$. We show that $D_{f,r}=0$ for every label $r\in[c]$
and derive a contradiction.

\emph{Labels in $[p]$.} Fix $i\in[p]$ and put $\Phi=D_{f,i}$. Then
$\Phi=\Phi^*$ and $\Phi\ge0$ by Lemma~\ref{lem:effects}(c). Let
$b\in\Gamma$. Since $b\notin[p]$, we have
$L(f)\cap L(h_{i,b})=[p]\cap\{i,b\}=\{i\}$, and $f\sim h_{i,b}$ by
\ref{E:fh}. Lemma~\ref{lem:effects}(e) gives $D_{h_{i,b},i}\Phi=0$, and
Lemma~\ref{lem:effects}(d) gives
\[
 D_{h_{i,b},b}\Phi=(I-D_{h_{i,b},i})\Phi=\Phi .
\]
Since $i\notin\Lambda_i\cup\{b\}$, we have
$L(z_{i,b})\cap L(h_{i,b})=\{b\}$, and $h_{i,b}\sim z_{i,b}$ by \ref{E:hz}.
Lemma~\ref{lem:effects}(e) gives $D_{z_{i,b},b}D_{h_{i,b},b}=0$, and
therefore
\[
 D_{z_{i,b},b}\Phi=D_{z_{i,b},b}D_{h_{i,b},b}\Phi=0 .
\]
Consequently, by Lemma~\ref{lem:effects}(c) and $\Phi=\Phi^*$,
\[
 0=\Phi D_{z_{i,b},b}\Phi
  =\sum_a\bigl(P_{z_{i,b},a}Q_{b,a}\Phi\bigr)^*
         \bigl(P_{z_{i,b},a}Q_{b,a}\Phi\bigr),
\]
and \ref{C:sumzero} and \ref{C:cstar} give
\begin{equation}\label{eq:z-kill}
 P_{z_{i,b},a}Q_{b,a}\Phi=0\qquad(a\in[c],\ b\in\Gamma).
\end{equation}
We claim that
\begin{equation}\label{eq:restricted}
 P_{v,a}\Phi=P_{v,a}Q_{\Lambda_i,a}\Phi\qquad(v\in C_i,\ a\in[c]).
\end{equation}
For $v=\ell_{i,o}$ this is Lemma~\ref{lem:effects}(b), because
$L(v)=\Lambda_i$. For $v=z_{i,b}$ we have $L(v)=\Lambda_i\cup\{b\}$ with
$b\notin\Lambda_i$, so Lemma~\ref{lem:effects}(b) gives
$P_{v,a}=P_{v,a}Q_{\Lambda_i,a}+P_{v,a}Q_{b,a}$, and \eqref{eq:z-kill}
removes the last term after multiplication by $\Phi$. From
\eqref{eq:restricted} and $P_{v,a}^2=P_{v,a}$ we obtain
\begin{equation}\label{eq:compressed}
 \Phi P_{v,a}\Phi=(P_{v,a}\Phi)^*(P_{v,a}\Phi)
  =\Phi Q_{\Lambda_i,a}P_{v,a}Q_{\Lambda_i,a}\Phi
  \qquad(v\in C_i,\ a\in[c]).
\end{equation}
No commutation relation between $\Phi$ and the other operators is used.

Fix $a$. Since $C_i$ is a clique, the projections $P_{v,a}$, $v\in C_i$, are
pairwise orthogonal, so $\sum_{v\in C_i}P_{v,a}\le I$ by \ref{C:proj}. By
Lemma~\ref{lem:effects}(a), $Q_{\Lambda_i,a}$ is a projection, and
$\sum_aQ_{\Lambda_i,a}=\sum_{r\in\Lambda_i}\sum_aQ_{r,a}=2tI$. In the
following computation the two equalities in the first line use
$\sum_aP_{v,a}=I$ and \eqref{eq:compressed}. The inequality is \ref{C:conj},
applied to $\sum_{v\in C_i}P_{v,a}\le I$ and the element
$Q_{\Lambda_i,a}\Phi$, together with $Q_{\Lambda_i,a}^2=Q_{\Lambda_i,a}$. We
get
\begin{align*}
 (2t+1)\Phi^2
 &=\sum_{v\in C_i}\sum_a\Phi P_{v,a}\Phi
  =\sum_a\Phi Q_{\Lambda_i,a}
     \Bigl(\sum_{v\in C_i}P_{v,a}\Bigr)Q_{\Lambda_i,a}\Phi\\
 &\le\sum_a\Phi Q_{\Lambda_i,a}\Phi=2t\Phi^2 .
\end{align*}
Hence $-\Phi^2\ge0$. Also $\Phi^2=\Phi^*\Phi\ge0$, so $\Phi^2=0$ by
\ref{C:cone} and $\Phi=0$ by \ref{C:cstar}. Thus $D_{f,i}=0$ for every
$i\in[p]$.

\emph{Labels outside $[p]$.} Let $r\in[c]\setminus[p]$. Then
$r\notin L(f)$, so $f\sim e_r$ by \eqref{eq:list-rule}, and
$P_{f,a}Q_{r,a}=0$ for every $a$. Hence $D_{f,r}=0$.

\emph{Contradiction.} For all $r\in[c]$ we have
$0=D_{f,r}=\sum_a(P_{f,a}Q_{r,a})^*(P_{f,a}Q_{r,a})$ by
Lemma~\ref{lem:effects}(c), so
$P_{f,a}Q_{r,a}=0$ for all $a,r$ by \ref{C:sumzero} and \ref{C:cstar}. With
Lemma~\ref{lem:effects}(a),
\[
 P_{f,a}=P_{f,a}\sum_{r\in[c]}Q_{r,a}=0\qquad(a\in[c]),
\]
which contradicts $\sum_aP_{f,a}=I\ne0$.

\emph{An ordinary $(c+1)$-coloring.} We use the labels and one new color $c$
as colors. Give every anchor $e_r$ the color $r$, $f$ the color $c$,
$h_{i,b}$ the color $i$, $z_{i,b}$ the color $b$, and $\ell_{i,o}$ the color
$o$. Every vertex other than $f$ has a color in its list, and the color of
$f$ is used only once; this settles \ref{E:anchors}, \ref{E:lists} and
\ref{E:fh}. For \ref{E:hz} note that $i\ne b$. For \ref{E:clique}, the colors
$o\in\cO_i$ of the vertices $\ell_{i,o}$ and the colors $b\in\Gamma$ of the
vertices $z_{i,b}$ are pairwise different, because
$\cO_i\subseteq\Lambda$ is disjoint from $\Gamma$.

For the last assertion, recall from Section~\ref{sec:intro} that a
$k$-coloring with
$k\le c$ extends to a $c$-coloring in the same algebra. Hence $H_t$ has no
$k$-coloring in a nonzero unital $C^*$-algebra for any $k\le c$, that is,
$\cstar(H_t)\ge c+1$. Now $\cstar\le\cq\le\chi$ and $\chi(H_t)\le c+1$ give
the equalities.
\end{proof}

\section{The product coloring}\label{sec:product}

Let $t\ge1$ be an integer, and let $p=2t$ and $c=3t+2$. Let $F$ be a graph
with vertex set $[p]$ that has no closed walk of length three or five, and let
$G=F[K_t]$ and $H=H_t$. We use the notation of Section~\ref{sec:gadget}. For
$i\in[p]$ the set $V(F)$ is the disjoint union of the sets
\[
 E_i=\{v:\dist_F(i,v)\in\{0,2\}\},\qquad
 O_i=\{v:\dist_F(i,v)=1\},\qquad
 T_i=V(F)\setminus(E_i\cup O_i),
\]
where $\dist_F(i,v)=\infty$ if $i$ and $v$ lie in different components.
Thus $O_i$ is the set of neighbors of $i$. The sets $O_i$ and $T_i$ may be
empty.

\begin{lemma}\label{lem:regions}
Let $i\in[p]$.
\begin{enumerate}[(a)]
\item $E_i$ and $O_i$ are independent sets of $F$.
\item Every neighbor of a vertex in $\{i\}\cup O_i$ lies in $E_i\cup O_i$.
\item There is no edge between $O_i$ and $T_i$.
\end{enumerate}
\end{lemma}
\begin{proof}
(a) An edge $uv$ with $u,v\in O_i$ gives the closed walk $i,u,v,i$ of length
three. Let $u,v$ have distance two from $i$, with paths $i,w,u$ and $i,w',v$.
An edge $uv$ gives the closed walk $i,w,u,v,w',i$ of length five. The
neighbors of $i$ lie in $O_i$, so $i$ has no neighbor in $E_i$.
(b) A neighbor of $i$ lies in $O_i$. A neighbor of a vertex in $O_i$ has
distance $0$, $1$ or $2$ from $i$ and hence lies in $E_i\cup O_i$.
(c) follows from (b).
\end{proof}

We define functions $\psi_y\colon V(G)\to[c]$ for $y\in V(H)$. Let
$(v,j)\in V(G)$, so $v\in[p]$ and $j\in[t]$. Put
\begin{align*}
 \psi_{e_r}(v,j)&=r,&\psi_f(v,j)&=v,&
 \psi_{h_{i,b}}(v,j)&=
 \begin{cases}b,&v\in\{i\}\cup O_i,\\i,&\text{otherwise}.\end{cases}
\end{align*}
For $y\in C_i$ let $\lambda_y=o$ if $y=\ell_{i,o}$ and $\lambda_y=b$ if
$y=z_{i,b}$, and put
\begin{equation}\label{eq:evaluation}
 \psi_y(v,j)=
 \begin{cases}
  \rho_i(j),&v\in E_i,\\
  \rho_i(t+j),&v\in O_i,\\
  \lambda_y,&v\in T_i.
 \end{cases}
\end{equation}
Thus a vertex $y\in C_i$ uses the labels in $\cE_i$ on $E_i$, the labels in
$\cO_i$ on $O_i$, and the single label $\lambda_y\in\cO_i\cup\Gamma$ on
$T_i$. The map $y\mapsto\lambda_y$ is injective on $C_i$. Moreover
$\lambda_y\in L(y)$: for $y=\ell_{i,o}$ we have
$o\in\cO_i\subseteq\Lambda_i=L(y)$, and for $y=z_{i,b}$ we have
$b\in\Lambda_i\cup\{b\}=L(y)$. Also $\cE_i\cup\cO_i=\Lambda_i\subseteq L(y)$
for every $y\in C_i$, and the values of $\psi_{e_r}$, $\psi_f$ and
$\psi_{h_{i,b}}$ lie in $\{r\}$, $[p]$ and $\{i,b\}$. Hence in all cases
\begin{equation}\label{eq:in-list}
 \psi_y(x)\in L(y)\qquad(y\in V(H),\ x\in V(G)).
\end{equation}

\begin{proposition}\label{prop:product}
The map $(x,y)\mapsto\psi_y(x)$ is an ordinary $c$-coloring of $G\times H$.
In particular $\chi(G\times H)\le c$.
\end{proposition}
\begin{proof}
We have to show that
\begin{equation}\label{eq:cross-condition}
 \psi_y(x)\ne\psi_{y'}(x')\qquad
 \text{for all }x\sim x'\text{ in }G\text{ and all }y\sim y'\text{ in }H.
\end{equation}
The relation $x\sim x'$ is symmetric. Hence it suffices to verify
\eqref{eq:cross-condition} for one ordering $(y,y')$ of each edge of $H$ and
all ordered pairs $(x,x')$ with $x\sim x'$. Write $x=(v,j)$ and
$x'=(v',j')$. Then $v\sim v'$, or $v=v'$ and $j\ne j'$.

\emph{Edges \ref{E:anchors} and \ref{E:lists}.} Let $y'=e_r$ and let $y$ be
adjacent to $e_r$. Then $r\notin L(y)$ by \eqref{eq:list-rule}. By
\eqref{eq:in-list}, $\psi_y(x)\in L(y)$, whereas $\psi_{e_r}(x')=r$.

\emph{Edges \ref{E:fh}.} Let $y=f$ and $y'=h_{i,b}$. The values of $\psi_f$
lie in $[p]$ and those of $\psi_{h_{i,b}}$ in $\{i,b\}$, where
$b\notin[p]$. So equality in \eqref{eq:cross-condition} requires
$\psi_f(x)=i$, that is, $v=i$. Then $v'\in\{i\}\cup O_i$, and
$\psi_{h_{i,b}}(x')=b\ne i$.

\emph{Edges \ref{E:hz}.} Let $y=h_{i,b}$ and $y'=z_{i,b}$. The values of
$\psi_{z_{i,b}}$ lie in $\Lambda_i\cup\{b\}$, which does not contain $i$.
So equality requires $\psi_{h_{i,b}}(x)=b$, that is, $v\in\{i\}\cup O_i$.
If $v'=v$, then $v'\in\{i\}\cup O_i\subseteq E_i\cup O_i$; if $v'\sim v$,
then $v'\in E_i\cup O_i$ by Lemma~\ref{lem:regions}(b). In both cases
$\psi_{z_{i,b}}(x')\in\cE_i\cup\cO_i=\Lambda_i$, which does not contain $b$.

\emph{Edges \ref{E:clique}.} Let $y\ne y'$ in $C_i$. The following cases are
exhaustive up to exchanging $(x,y)$ and $(x',y')$.
\begin{center}
\begin{tabular}{@{}lp{.72\linewidth}@{}}
\toprule
Sets containing $v,v'$ & Reason for $\psi_y(x)\ne\psi_{y'}(x')$\\
\midrule
$E_i,E_i$ & $E_i$ is independent, so $v=v'$ and $j\ne j'$; then
$\rho_i(j)\ne\rho_i(j')$ because $\rho_i$ is injective.\\
$O_i,O_i$ & $O_i$ is independent, so $v=v'$ and $j\ne j'$; then
$\rho_i(t+j)\ne\rho_i(t+j')$.\\
$E_i,O_i$ & $\cE_i$ and $\cO_i$ are disjoint.\\
$E_i,T_i$ & $\cE_i$ and $\cO_i\cup\Gamma$ are disjoint.\\
$O_i,T_i$ & This case does not occur: $v\ne v'$, so $v\sim v'$, which
contradicts Lemma~\ref{lem:regions}(c).\\
$T_i,T_i$ & $\lambda_y\ne\lambda_{y'}$.\\
\bottomrule
\end{tabular}
\end{center}
This covers all edges of $H$.
\end{proof}

\section{Proofs of the theorems}\label{sec:proofs}

\begin{proof}[Proof of Theorem~\ref{thm:general}]
Part (a) is Proposition~\ref{prop:product}. For part (b), a $c$-coloring of
$G=F[K_t]$ in a nonzero unital $C^*$-algebra would give
$t\su(M)\le c\tr M$ by Corollary~\ref{cor:blowup}, contradicting
\eqref{eq:theta-hypothesis}. Part (c) is Proposition~\ref{prop:H}.

Let $\pi$ be a graph parameter with $\cstar\le\pi\le\chi$. Then
$\pi(G\times H)\le\chi(G\times H)\le c$ by (a). By (b) and (c), and because a
$k$-coloring with $k\le c$ extends to a $c$-coloring, we have
$\cstar(G)\ge c+1$ and $\cstar(H)\ge c+1$. Hence $\pi(G)\ge c+1$ and
$c+1\le\pi(H)\le\chi(H)\le c+1$.
\end{proof}

\begin{proof}[Proof of Theorem~\ref{thm:main}]
By Lemma~\ref{lem:certificate}, the graph $F$ of Section~\ref{sec:base} and
the matrix $M=A+10I$ satisfy the hypotheses of Theorem~\ref{thm:general}
with $t=512$, $p=1024$ and $c=1538$: indeed
$t\su(M)=16384p>15380p=c\tr M$. Hence Theorem~\ref{thm:general} gives (a),
the first assertion of (b), and \eqref{eq:main}. The last two sentences of
the theorem follow from \eqref{eq:main} and \eqref{eq:hierarchy}. The orders
of $G$ and $H$ are computed after \eqref{eq:fibers} and in
\eqref{eq:Horder}. The bound $\cstar(G)\ge1639$ is
Corollary~\ref{cor:Gbound}, and $\cstar(H)=\chi(H)=1539$ is
Proposition~\ref{prop:H}.
\end{proof}

\section{Concluding remarks}\label{sec:conclusion}

The conjecture \eqref{eq:conjecture} is equivalent to the statement that,
for every integer $c\ge1$ and all graphs $G,H$, the inequality
$\cq(G\times H)\le c$ implies $\min\{\cq(G),\cq(H)\}\le c$.
Theorem~\ref{thm:main} and Appendix~\ref{sec:smaller} show that this
implication fails for $c=1538$ and for $c=1184$. It holds for $c\le2$.
Indeed, $\cq(X)\le1$ if and only if $X$ has no edge, and $G\times H$ has an
edge if and only if both $G$ and $H$ have an edge. Moreover, $\cq(X)\le2$ if
and only if $\chi(X)\le2$ \cite[Proposition~3]{CMNSW}, and $G\times H$ is
bipartite if and only if $G$ or $H$ is bipartite. For the nontrivial
implication, recall that a graph is bipartite if and only if it has no closed
walk of odd length (Remark~\ref{rem:hypotheses}(i)). If $G$ and $H$ have
closed walks $x_0,\ldots,x_m=x_0$ and $y_0,\ldots,y_n=y_0$ of odd lengths $m$
and $n$, then the vertices $(x_{r\bmod m},y_{r\bmod n})$, $0\le r\le mn$, form
a closed walk of odd length $mn$ in $G\times H$. For Hedetniemi's
conjecture the case $c=3$ is a theorem of El-Zahar and Sauer~\cite{EZS}.

\begin{question}\label{q:three}
Does $\cq(G\times H)\le3$ imply $\min\{\cq(G),\cq(H)\}\le3$ for all graphs
$G,H$? What is the least $c$ for which there are graphs $G,H$ with
$\cq(G\times H)\le c<\min\{\cq(G),\cq(H)\}$?
\end{question}

\begin{question}\label{q:order}
What is the least $t$ for which a graph $F$ as in
Theorem~\ref{thm:general} exists? Equivalently, what is the least even $p$
for which there is a graph of order $p$ and odd girth at least $7$ with
$\thbar>3+4/p$?
\end{question}

The two formulations of Question~\ref{q:order} are equivalent by
Remark~\ref{rem:hypotheses}(i) and~(ii), with $p=2t$.
Appendix~\ref{sec:smaller} shows that the least such $p$ is at most $788$.
Question~\ref{q:order} is the analogue for $\thbar$ of
\cite[Question~6]{Zhu}.

\section*{Formal verification}

All theorems, corollaries, lemmas and propositions of this paper have been verified in Lean 4.19.0 with Mathlib, in the projective
formulation~\eqref{eq:relations} and for the parameters $\chi$, $\cq$ and
$\cstar$. Theorem~\ref{thm:general} is not stated as a single Lean theorem:
Lemma~\ref{lem:psd} and Propositions~\ref{prop:H} and~\ref{prop:product} are
verified in general form, and they are combined in Lean only for $t=512$ and
$t=394$. The equivalence of \eqref{eq:relations} with the coloring game and
the relations \eqref{eq:hierarchy} between the models are taken from the
literature and are not formalized. The Lean sources, a table that relates
them to the statements of this paper, the certificate data and exact Python
checkers are available in~\cite{zeiss-lean}.

\section*{A.I. disclosure}
\addcontentsline{toc}{section}{A.I. disclosure}
Generative AI was used extensively in the development
of this paper, including in exploring and refining proof ideas and in drafting and revising the text. In
particular, the proof idea underlying the main theorems and the construction of the counterexample is the result of prompting a generative AI system,
which proposed the core argument developed in this paper. The author subsequently verified, simplified,
and organized the argument, and edited its exposition. The author takes full responsibility for the correctness and
content of the paper.I am sharing these findings on arXiv, with a transparent account of the substantial contribution made by LLMs, to invite independent scrutiny and support further research.

\section*{Acknowledgments}
\addcontentsline{toc}{section}{Acknowledgments}
The author thanks Dorian Rudolph and Cormac Stopes for discussions on quantum graph coloring games,
and the Lean and Mathlib communities for the proof
assistant and mathematical library used in this work.
JZ acknowledges support from the European Research Council
(ERC Grant Agreement No.~948139) and the Excellence Cluster --
Matter and Light for Quantum Computing (ML4Q-2).

\appendix
\addtocontents{toc}{\protect\let\protect\numberline\protect\appendixnumberline}
\section[Appendix]{Smaller counterexamples}\label{sec:smaller}

The first example in this appendix is the instance of
Theorem~\ref{thm:general} for a base graph $F_0$ of order $788$. The second
example is obtained from the first by deleting $155$ vertices of the first
factor. The certificate for the hypothesis \eqref{eq:theta-hypothesis} and
the data for the second example are given by exact integers in the
directory \texttt{data/} of the repository~\cite{zeiss-lean}.

\subsection{The graphs}

Let $F$ be the graph of Section~\ref{sec:base}. Put
\begin{align*}
 \mathcal I&=\{0,2,4,8,23,24,30,34,40,42,46,50,53,61\},\\
 V_0&=\{v\in[1024]:\lfloor v/16\rfloor\notin\mathcal I\}
       \setminus\{112,\ldots,123\},
\end{align*}
and let $F_0$ be the subgraph of $F$ induced by $V_0$. It has $788$
vertices. To apply Sections~\ref{sec:gadget} and~\ref{sec:product} we
identify the vertex set $V_0$ of $F_0$ with $[788]$ through the increasing
bijection $[788]\to V_0$, $v\mapsto\nu_v$. As in the data files, we call
$\nu_v\in V_0$ the \emph{coordinate} of the vertex $v\in[788]$. Put
\[
 t_0=394,\qquad p_0=788=2t_0,\qquad c_0=3t_0+2=1184,\qquad
 H_0=H_{t_0},\qquad G_0=F_0[K_{t_0}].
\]
Then $|V(G_0)|=310472$, and $|V(H_0)|=789\cdot1184+1=934177$ by
\eqref{eq:Horder}. Let $v^*$ be the vertex of $F_0$ with $\nu_{v^*}=996$,
and let $G'_0$ be the subgraph of $G_0$ induced by all vertices except the
$155$ vertices $(v^*,j)$ with $239\le j<394$. Then $|V(G'_0)|=310317$.

\begin{corollary}\label{cor:smaller}
Let $X=G_0$ or $X=G'_0$. Then $\chi(X\times H_0)\le1184$, neither $X$ nor
$H_0$ has a $1184$-coloring in any nonzero unital $C^*$-algebra, and
$\chi(H_0)=1185$. Consequently, every graph parameter $\pi$ with
$\cstar\le\pi\le\chi$, in particular each of $\cq$, $\chi_{qs}$,
$\chi_{qa}$, $\chi_{qc}$ and $\cstar$, satisfies
\[
 \pi(X\times H_0)\le1184<1185=\min\{\pi(X),\pi(H_0)\}.
\]
\end{corollary}

Since $G'_0$ is an induced subgraph of $G_0$, a $k$-coloring of $G_0$ in a
unital $C^*$-algebra restricts to a $k$-coloring of $G'_0$ in the same
algebra, and $G'_0\times H_0$ is an induced subgraph of $G_0\times H_0$.
Hence the nonexistence of a $1184$-coloring of $G_0$ in a nonzero unital
$C^*$-algebra follows from that of $G'_0$, and
$\chi(G'_0\times H_0)\le1184$ follows from $\chi(G_0\times H_0)\le1184$. We
prove the statement for $G_0$ separately, because its certificate is
simpler.

\subsection{The certificate}

\begin{lemma}\label{lem:small-certificate}
There is a real positive-semidefinite matrix $M_0$ with integer entries that
is supported on $F_0$ and satisfies
\begin{align}
 \tr M_0&=400025192,&\su(M_0)&=1202501792,\label{eq:totals}\\
 394\su(M_0)-1184\tr M_0&=155878720>0.&&\label{eq:small-gap}
\end{align}
\end{lemma}
\begin{proof}[Proof by exact certificate]
\emph{Definition of $M_0$.} The set $V_0$ is a union of $197$ cosets of the
subgroup $\{0,1,2,3\}$ of $\F_2^{10}$. We index the vertex with coordinate
$\nu$ by the pair $(i,x)$, where $i\in[197]$ is the position of
$\lfloor\nu/4\rfloor$ among the values $\lfloor\nu'/4\rfloor$,
$\nu'\in V_0$, in increasing order, and $x=(x_0,x_1)\in\F_2^2$ consists of the two low
bits of $\nu$, that is, $\nu\equiv x_0+2x_1\pmod 4$. In the data an element
$z=(z_0,z_1)\in\F_2^2$ is written as the integer $z_0+2z_1$. The file \texttt{matrix\_certificate.json} contains the list
\texttt{cosets} of the $197$ values $\lfloor\nu/4\rfloor$ in increasing
order, and the set $S$ as \texttt{steps}. The group $\{0,1,2,3\}$ acts on $F_0$ by
the translations $\nu\mapsto\nu\oplus g$. The field \texttt{specs} lists
its $197$ orbits on the vertices and its $1992$ orbits on the edges of
$F_0$, and \texttt{integer\_weights} lists the integer weights of these
$2189$ orbits; the fields \texttt{rank} and \texttt{basis} are not needed
here. Start with the zero
matrix. An orbit $(\texttt{diag},i,i,0)$ of weight $m$ adds $m$ to the entry
$((i,x),(i,x))$ for each $x\in\F_2^2$. An orbit $(\texttt{edge},i,j,z)$ with
$i\le j$ and weight $m$ adds $m$ to the entry $((i,x),(j,x\oplus z))$ for
each $x\in\F_2^2$ and, if $i\ne j$, also to the transposed entry. The result
is $M_0$. One checks entrywise that $M_0$ is symmetric and supported on
$F_0$, and that
\begin{equation}\label{eq:kernel}
 (M_0)_{(i,x),(j,y)}=\cK_{x\oplus y}(i,j)
\end{equation}
for four real symmetric $197\times197$ matrices $\cK_z$, $z\in\F_2^2$.

\emph{Reduction to four blocks.} For $a,z\in\F_2^2$ let
$a\cdot z=a_0z_0+a_1z_1\in\F_2$, and put
\begin{equation}\label{eq:blocks}
 \widehat{\cK}_a=\sum_{z\in\F_2^2}(-1)^{a\cdot z}\cK_z\qquad(a\in\F_2^2).
\end{equation}
For a real vector $\xi=(\xi_{i,x})$ define
$\widehat\xi^{(a)}_i=\frac12\sum_{x\in\F_2^2}(-1)^{a\cdot x}\xi_{i,x}$.
Since $\sum_{x\in\F_2^2}(-1)^{(a\oplus b)\cdot x}=4\delta_{a,b}$ for
$a,b\in\F_2^2$, and since $a\cdot x=x\cdot a$, we have
$\xi_{i,x}=\frac12\sum_a(-1)^{a\cdot x}\widehat\xi^{(a)}_i$. With
\eqref{eq:kernel} this gives
\begin{align*}
 \xi^{\mathsf T}M_0\xi
 &=\sum_{i,j}\sum_{x,y}\xi_{i,x}\,\cK_{x\oplus y}(i,j)\,\xi_{j,y}\\
 &=\frac14\sum_{a,b}\sum_{i,j}\widehat\xi^{(a)}_i\,\widehat\xi^{(b)}_j
   \sum_{x,y}(-1)^{a\cdot x+b\cdot y}\cK_{x\oplus y}(i,j).
\end{align*}
In the inner sum put $y=x\oplus z$, so that
$b\cdot y=b\cdot x+b\cdot z$. Then
\[
 \sum_{x,y}(-1)^{a\cdot x+b\cdot y}\cK_{x\oplus y}
 =\sum_z(-1)^{b\cdot z}\cK_z\sum_x(-1)^{(a\oplus b)\cdot x}
 =4\delta_{a,b}\widehat{\cK}_a ,
\]
and therefore
\[
 \xi^{\mathsf T}M_0\xi
 =\sum_{a\in\F_2^2}(\widehat\xi^{(a)})^{\mathsf T}\widehat{\cK}_a\,
   \widehat\xi^{(a)} .
\]
Hence $M_0\succeq0$ if the four matrices $\widehat{\cK}_a$ are positive
semidefinite.

\emph{Positivity of the blocks.} The file
\texttt{gram\_certificate.json.gz} contains the integer $q=10^6$ as
\texttt{scale} and, as \texttt{lower\_triangular\_integer\_factors}, a list
of four integer lower-triangular matrices $L_a$, where row $r$ of a matrix
is stored as the list of its entries in columns $0,\ldots,r$;
the matrix at position $a_0+2a_1\in\{0,1,2,3\}$ of the list belongs to
$a=(a_0,a_1)$. Put
\[
 \Delta_a=q^2\widehat{\cK}_a-L_aL_a^{\mathsf T}.
\]
An exact integer computation shows that, for all four blocks and all rows
$i$,
\begin{equation}\label{eq:dd}
 (\Delta_a)_{ii}\ge\sum_{j\ne i}|(\Delta_a)_{ij}|.
\end{equation}
A real symmetric matrix $\Delta$ with \eqref{eq:dd} is positive
semidefinite, because
$2|\Delta_{ij}\zeta_i\zeta_j|\le|\Delta_{ij}|(\zeta_i^2+\zeta_j^2)$ gives
\[
 \zeta^{\mathsf T}\Delta\zeta\ge
 \sum_i\Bigl(\Delta_{ii}-\sum_{j\ne i}|\Delta_{ij}|\Bigr)\zeta_i^2\ge0 .
\]
Consequently
$\widehat{\cK}_a=q^{-2}(L_aL_a^{\mathsf T}+\Delta_a)\succeq0$ and
$M_0\succeq0$.

\emph{Totals.} Exact summation gives \eqref{eq:totals}, and
\eqref{eq:small-gap} follows.

\emph{Formalization.} The Lean sources contain the entries of the matrices
$\cK_z$, $\widehat{\cK}_a$ and $L_a$ as integer data, and $M_0$ is defined
there by \eqref{eq:kernel}. For these data Lean proves that $M_0$ is
supported on $F_0$, the identities \eqref{eq:blocks}, the inequalities
\eqref{eq:dd}, that $M_0\succeq0$, and \eqref{eq:totals}. Lean does not read
the JSON files. That the data embedded in the Lean sources agree with the
data obtained from the two JSON files as described above is not a theorem
of Lean; it is checked by an exact comparison, the script
\texttt{numerics/compare\_lean\_data.py} of the
repository~\cite{zeiss-lean}.
Independently of Lean, the Python checker \texttt{data/verify.py} performs
all checks of this proof on the JSON files. Neither Lean nor the Python
checker uses floating-point arithmetic.
\end{proof}

\subsection{Proof of Corollary~\ref{cor:smaller}}

\begin{proof}[Proof for $X=G_0$]
The graph $F_0$ has no closed walk of length three or five, because such a
walk would be a walk in $F$, contradicting
Lemma~\ref{lem:certificate}(a). By Lemma~\ref{lem:small-certificate}, the
matrix $M_0$ satisfies \eqref{eq:theta-hypothesis} with $t=t_0$ and
$c=c_0$. Hence Theorem~\ref{thm:general} applies to $F_0$, with the above
identification of its vertex set with $[788]$, and the corollary for $G_0$
follows from
Theorem~\ref{thm:general} and Proposition~\ref{prop:H}.
\end{proof}

\begin{proof}[Proof for $X=G'_0$]
The graph $G'_0\times H_0$ is a subgraph of $G_0\times H_0$, so
$\chi(G'_0\times H_0)\le1184$. The statements about $H_0$ were proved
above. It remains to show that $G'_0$ has no $1184$-coloring in a nonzero
unital $C^*$-algebra.

For $v\in[788]$ let $s_v=239$ if $v=v^*$ and $s_v=394$ otherwise, so
that $V(G'_0)=\{(v,j):v\in[788],\ 0\le j<s_v\}$. Let $w_v=74$ for all
$v\in[788]$ except those whose coordinate $\nu_v$ is listed in
Table~\ref{tab:weights}, for which $w_v$ is the value given there. The file
\texttt{instance.json} contains these data: \texttt{fibers} gives the
numbers $s_v$, and \texttt{diagonal\_congruence} gives the numbers $w_v$; in
both fields the exceptions are indexed by the coordinate $\nu_v$. Define the
matrix $\widetilde M$ indexed by $V(G'_0)$ by
\[
 \widetilde M_{(u,j),(v,\ell)}=w_u(M_0)_{uv}w_v .
\]
Then $\widetilde M=W^{\mathsf T}M_0W$ for the real $788\times310317$ matrix
$W$ with entries $W_{u,(v,j)}=w_v\delta_{u,v}$, so $\widetilde M\succeq0$.
If two distinct vertices $(u,j),(v,\ell)$ of $G'_0$ are not adjacent, then
$u\ne v$ and $u\not\sim v$ in $F_0$, so $(M_0)_{uv}=0$. Thus $\widetilde M$
is supported on $G'_0$. Exact integer summation gives
\begin{align*}
 \su(\widetilde M)&=\sum_{u,v}s_us_vw_u(M_0)_{uv}w_v
   =1020438226912198088,\\
 \tr\widetilde M&=\sum_vs_vw_v^2(M_0)_{vv}=861856598314848,
\end{align*}
and $\su(\widetilde M)-1184\tr\widetilde M=14507418056>0$. By
Lemma~\ref{lem:psd}, $G'_0$ has no $1184$-coloring in a nonzero unital
$C^*$-algebra. The inequalities for $\pi$ follow as in the proof of
Theorem~\ref{thm:general}.
\end{proof}

\begin{table}[ht]
\centering
\begin{tabular}{@{}l*{10}{r}@{}}
\toprule
$\nu_v$&26&185&193&291&349&351&587&663&785&868\\
$w_v$&72&72&73&73&73&75&73&73&72&72\\
\midrule
$\nu_v$&932&964&992&994&996&998&1004&1006&1012&1019\\
$w_v$&72&72&69&75&60&75&73&75&72&72\\
\bottomrule
\end{tabular}
\caption{The twenty vertices $v$ of $F_0$, given by their coordinates
$\nu_v$, with $w_v\ne74$.}
\label{tab:weights}
\end{table}

\end{document}